\documentclass[a4paper, 11pt]{article}

\usepackage{longtable}
\usepackage{geometry}
\usepackage{graphicx}
\usepackage{subfig}
\usepackage{color}
\usepackage{xcolor}
\usepackage{bm}
\usepackage{fixmath}
\usepackage{mathrsfs}
\usepackage{ulem}
\usepackage{multirow}
\usepackage{pdflscape}
\usepackage{rotating}
\usepackage{lscape}
\usepackage{pdflscape}
\usepackage{titlesec}
\usepackage{float}
\usepackage{hyperref}
\usepackage{tabularx}
\usepackage{mathabx}
\usepackage{wasysym}
\usepackage[utf8]{inputenc}
\usepackage[english]{babel}
\usepackage{amsthm}

\newtheorem{theorem}{Theorem}

\def\b{\begin{eqnarray}}
\def\e{\end{eqnarray}}
\def\n{\noindent}

\begin{document}

\begin{center}
{\huge \textbf{On the Cubic Equation with its  \vskip.2cm Siebeck--Marden--Northshield  Triangle  \vskip.1cm and the Quartic Equation with its \vskip.3cm Tetrahedron}}

%    On the Cubic Equation with its Siebeck--Marden--Northshield Triangle and the Quartic Equation with its Tetrahedron

\vspace {12mm}
\noindent
{\Large \bf Emil M. Prodanov}
\vskip10mm
{\it School of Mathematical Sciences, Technological University Dublin,} \\
{\it Park House, Grangegorman, 191 North Circular Road, } \\
{\it Dublin D07 EWV4, Ireland} \\
{\it E-Mail: emil.prodanov@tudublin.ie}
\vskip1cm

\end{center}

\vskip1cm

\noindent
\begin{abstract}
\noindent
The real roots of the cubic and quartic polynomials are studied geometrically with the help of their respective Siebeck--Marden--Northshield equilateral triangle and regular tetrahedron. The Vi\`ete trigonometric formul\ae\, for the roots of the cubic are established through the rotation of the triangle by variation of the free term of the cubic. A very detailed complete root classification for the quartic $x^4 + ax^3 + bx^2 + cx + d$ is proposed for which the conditions are imposed on the individual coefficients $a$, $b$, $c$, and $d$. The maximum and minimum lengths of the interval containing the four real roots of the quartic are determined in terms of $a$ and $b$. The upper and lower root bounds for a quartic with four real roots are also found: no root can lie farther than $(\sqrt{3}/4) \, \sqrt{3a^2 - 8b}$ from $-a/4$. The real roots of the quartic are localized by finding intervals containing at most two roots. The end-points of these intervals depend on $a$ and $b$ and are roots of quadratic equations --- which makes this localization helpful for quartic equations with complicated parametric coefficients.
\end{abstract}

% The real roots of the cubic and quartic polynomials are studied geometrically with the help of their respective Siebeck--Marden--Northshield equilateral triangle and regular tetrahedron. The Viete trigonometric formulae for the roots of the cubic are established through the rotation of the triangle by variation of the free term of the cubic. A very detailed complete root classification for the quartic $x^4 + ax^3 + bx^2 + cx + d$ is proposed for which the conditions are imposed on the individual coefficients $a$, $b$, $c$, and $d$. The maximum and minimum lengths of the interval containing the four real roots of the quartic are determined in terms of $a$ and $b$. The upper and lower root bounds for a quartic with four real roots are also found: no root can lie farther than $(\sqrt{3}/4) \sqrt{3a^2 - 8b}$ from $-a/4$. The real roots of the quartic are localized by finding intervals containing at most two roots. The end-points of these intervals depend on $a$ and $b$ and are roots of quadratic equations --- which makes this localization helpful for quartic equations with complicated parametric coefficients.

\vskip1cm
\noindent
{\bf Mathematics Subject Classification Codes (2020)}: 12D10, 26C10, 26D05
\vskip1cm
\noindent
{\bf Keywords}: Cubic and quartic polynomials, Siebeck--Marden--Northshield triangle, \linebreak tetrahedron, discriminants, complete root classification, root localization.

%    Cubic and quartic polynomials; Siebeck--Marden--Northshield triangle; tetrahedron; discriminants; complete root classification; root localization.

\newpage

\section{Introduction}

\n
Until the XIX$^{\mbox{\scriptsize th}}$ century, algebra was, effectively, nothing else but theory of polynomial equations, reflecting the immense importance of the research on the problem of determination of the roots of an equation. Cubic equations have been known since Classical Antiquity. The geometric problem of doubling the cube, the so-called Delian problem, seen in book VII of Plato's {\it Republic}, amounts to solving $x^3 - 2 = 0$. The first recorded mention of quartic equations is in Pacioli’s {\it Summa de Arithmetica, Geometria, Proportioni e Proportionalità} from 1494, while Cardano's {\it Artis Magn\ae,  Sive de Regvlis Algebraicis, Liber Unus} from 1545 shows how to find the roots of cubic and quartic equations.  The applications of cubic and quartic equations in all branches of science are vast. For example, the description of a real gas with the van der Waals model is achieved through a cubic polynomial in the volume and the richer structure of the cubic allows the accommodation of phase transitions. There are well over 200 real gas equations, many of which are also cubic. The elastic waves propagating on the surface of solids, the so-called Rayleigh waves, are also described by a cubic equation (for the velocity profile). Quartic and quintic polynomials are used to represent cost functions in economics. The Hodgkin--Huxley model in mathematical neuroscience encounters a quartic for the quantitative description of membrane currents, conduction, and excitation in nerves. Quartic equations need to be solved for the description of the velocities in equatorial wave–current interactions. In general relativity, through the d'Inverno and Russel--Clark algorithms, the Petrov classification of the Weyl conformal curvature and the Pleba\'nski or Segre classification of the Ricci tensor can be achieved by the classification of the roots of a quartic equation whose coefficients are functions of the space-time coordinates --- through the components of the Riemann curvature tensor. \\
Roots of polynomials could be sought by the use of bracketing methods, yielding an upper limit on the number of real roots (such as Descartes' rule of signs and Newton's rule of signs) or the number of real roots on a prescribed interval (through Budan's theorem and Sturm's theorem). These methods could be used for the real-root isolation of polynomials: finding a set of intervals so that only one root is contained in each of them. The bisection and regula falsi methods also fall under this category as with these a root can be spotted on an interval, while interpolation and iterative methods (such as the Newton--Raphson method or the secant method) seek convergence to a root by the repeated application of a particular technique. \\
The roots of cubic and quartic polynomials with integer coefficients can be easily found with methods of numerical analysis or using strong computing power. However, these fail when the polynomials contain parameters, that is, for symbolic polynomials. Presented in this work is a geometrical study of the roots of the symbolic cubic and symbolic quartic polynomials with which new features are revealed. This is a new bracketing method that allows one to find an improved {\it Complete Root Classification} of the quartic, that is, determine first all possible cases of the roots (real and complex) and list their multiplicities (the so-called {\it Root Classification}), followed by the determination of the conditions which the coefficients of the equation must satisfy for each case of the root classification. Geometrical methods for the study of the roots of the symbolic cubic and symbolic quartic include the ``splitting" of the polynomial into two parts and studying the ``interaction" between them, namely, the intersection points of their curves \cite{23}--\cite{28}. The continuous variation of one of the equation coefficients allows one to trace how the discriminant of the polynomial is affected and, through this, determine the type of the roots and localize them --- also leading to Complete Root Classification. In \cite{27}, through the continuous variation of the free term of a ``split" cubic, a very detailed Complete Root Classification of the symbolic cubic is obtained and the isolation interval of each root is found in terms of simple functions of the coefficients of the cubic or the roots of simple quadratic equations, while in \cite{26}, through a similar variation of the free term of a symbolic quartic, a two-tier root classification is obtained, together with the determination of the root isolation intervals, by solving some auxiliary cubic or quadratic equations. The root classification for the quartic is also based on simple conditions on the individual coefficients and the end-points of the root isolation intervals are roots of cubic or quadratic polynomials. Rees \cite{rees} shows that the discriminant of a monic quartic is 256 times the product of the ordinates of the turning points of the graph and ``splitting" the depressed cubic as $x^4 + q x^2 + s = - r x$ allows one to obtain a Complete Root Classification by realizing that the turning points of the original quartic correspond to those points of the curve on the left-hand side at which the tangent is parallel to the line $y = - r x$. However, the conditions for each case of the Root Classification involve the determination of the sign of the quartic discriminant and in case of symbolic polynomials, this is not an easy task. \\
Vi\`ete \cite{viete} and Descartes \cite{descartes} first related the case of three real roots of a cubic polynomial (the {\it casus irreducibilis}) to circle geometry \cite{nickalls3}. The symbolic cubic polynomial $x^3 + ax^2 + bx + c$ with three real roots (not all equal) is associated with an equilateral triangle, the Siebeck--Marden--Northshield triangle \cite{north}--\cite{29}, the vertices of which have $x$-coordinate projections equal to the roots of the cubic. The centroid of the triangle is at the $x$-coordinate projection of the inflection point $-a/3$ of the cubic and the inscribed circle crosses the abscissa at the $x$-coordinate projections of the stationary points of the cubic. The radius of the inscribed circle is $(2/3) \sqrt{a^2 - 3b}$, hence a necessary condition for the existence of the Siebeck--Marden--Northshield triangle (or three real roots of the cubic) is $b < a^2/3$. If, additionally, the free term $c$ of the polynomial is between the roots of a specific quadratic equation, then one has a necessary and sufficient condition for three real roots of the cubic. This makes the study of the real roots easier than when considering if the cubic discriminant is positive. The variation of the free term results in rotation of the Siebeck--Marden--Northshield triangle and from that a relationship with Vi\`ete trigonometric formul\ae \, for the roots of the cubic is established. \\
On the other hand, a quartic polynomial with four real roots (not all equal) is associated with a regular tetrahedron and its symmetries \cite{north}, \cite{chalkley}--\cite{nickalls1}. The four real roots of the symbolic quartic $x^4 + a x^3 + b x^2 + c x +d$ are the $x$-coordinate projections of a regular tetrahedron in I \hskip-6.5pt R$^3$. The inscribed sphere of this tetrahedron projects onto an interval with endpoints given by the $x$ coordinates of the two inflection points of the quartic. The equilateral triangle in the $xy$-plane around that sphere is exactly the Siebeck--Marden--Northshield triangle for the cubic polynomial $4 x^3 + 3 a x^2 + 2 b x + c$ and its vertices project onto the $x$ coordinates of the stationary points of the quartic. The continuous variation of the free term of the quartic allows one to trace how the discriminant of the quartic changes sign and, from there, get the nature of the roots. This discriminant is cubic in the free term and the discriminant of this cubic is very simple. This allows the presentation of a very detailed complete root classification of the quartic with all possible cases listed (e.g. two pairs of equal real roots) and the conditions for these cases are conditions on the individual coefficients of the polynomial. Some further new results are also reported here: (i) the maximum and minimum lengths of the interval containing the four real roots of the quartic are determined --- the roots lie in an interval of length between $(\sqrt{2}/2) \sqrt{3a^2 - 8b}$ and $(\sqrt{3}/3) \sqrt{3a^2 - 8b}$; (ii) it is shown that a quartic with four real roots cannot have a root greater than $-a/4 + (\sqrt{3}/4) \, \sqrt{3a^2 - 8b}$ and cannot have a root smaller than $-a/4 - (\sqrt{3}/4) \, \sqrt{3a^2 - 8b}$; (iii) the real roots of the quartic are localized by finding intervals containing at most two roots and the end-points of these intervals depend on $a$ and $b$ and are roots of quadratic equations --- which makes this localization helpful for quartic equations with complicated parametric coefficients. The first steps towards the complete root classification for the quintic equation are also provided. \\

\section{The Siebeck--Marden--Northshield Triangle and Vi\`ete's Trigonometric Formul\ae \, for the Roots of the Cubic}

\subsection{The Siebeck--Marden--Northshield Triangle}

\n
Due to the Siebeck--Marden--Northshield theorem \cite{smn1}--\cite{29}, if the cubic polynomial $p_3(x)$ $= x^3 + a x^2 + b x + c$ has three real roots $x_1, \,\, x_2,$ and $x_3$, not all of which equal, these can be found geometrically as the $x$-coordinate projections of the vertices $(x_1, (x_2 - x_3)/\sqrt{3})$, $(x_2, (x_3 - x_1)/\sqrt{3})$, and $(x_3, (x_1 - x_2)/\sqrt{3})$ of an equilateral triangle --- the {\it Siebeck--Marden--Northshield triangle}. \\
The inscribed circle of this triangle is centered at $x = - a/3$, the $x$-coordinate projection of the inflection point of $p_3(x)$, while the intersection points of the inscribed circle with the $x$-axis are the $x$-coordinate projections of the critical points $\mu_{1,2} = -a/3 \pm (1/3) \sqrt{a^2 - 3b}$ of the cubic polynomial. Thus, the radius of the inscribed circle is $r = (1/3) \sqrt{a^2 - 3b}$ and the radius of the circumscribed circle is $2r = (2/3) \sqrt{a^2 - 3b}$. Each side $l$ of the Siebeck--Marden--Northshield triangle is equal to $(\sqrt{12}/3) \sqrt{a^2 - 3b}$ --- see Figure 1. \\
Instead of requesting a non-negative discriminant $\delta_3$ of the cubic polynomial to ensure three real roots, one can find another necessary and sufficient condition for their existence.
\begin{theorem} The necessary and sufficient condition for three real roots of the cubic $x^3 + ax^2 + bx + c$ is $a^2 - 3b > 0$ and $ - 2a^3/27 + ab/3 - (2/27) \sqrt{( a^2 - 3 b)^3} \le c \le - 2a^3/27 + ab/3 + (2/27) \sqrt{( a^2 - 3 b)^3}$ \cite{27, 29}.
\end{theorem}
\begin{proof}
Clearly, $a^2 - 3b > 0$ is a necessary condition for the existence of the Siebeck--Marden--Northshield triangle --- as this ensures that $p_3(x)$ will have two distinct critical points. But $p_3(x)$ will not necessarily have three real roots in this case. \\
Consider the discriminant
\b
\delta_3 = -27 c^2 + (18 a b - 4 a^3) c + a^2 b^2 - 4b^3
\e
of $p_3(x) = x^3 + a x^2 + b x + c$. \\
It is quadratic in $c$ and the discriminant of this quadratic is
\b
\delta_2 = 16 (a^2 - 3b)^3.
\e
As $a^2 - 3b > 0$, one has $\delta_2 > 0$ and thus $\delta_3(c) \ge 0$ for $c_2 \le c \le c_1$, where $c_{1,2}$ are the two distinct real roots (distinct, since $\delta_2 > 0$) of $\delta_3(c) = 0$:
\b
\label{c12}
c_{1,2}(a,b) = c_0 \, \pm \, \frac{2}{27} \, \sqrt{( a^2 - 3 b)^3},
\e
where
\b
\label{c0}
c_0(a,b) = -\frac{2}{27} \, a^3 + \frac{1}{3} \, a b.
\e
Hence, the cubic polynomial $p(x) = x^3 + a x^2 + b x + c$ with $a^2 - 3b > 0$ and $c_2 \le c \le c_1$ has three real roots. \\
Note that if $a^2 - 3b > 0$ and $c = c_{1,2}$, then the cubic will have a double real root and single real root. If $a^2 - 3b > 0$ and $c_2 < c < c_1$, the three real roots will be distinct. Finally, if $a^2 = 3b$, then the cubic will have a triple real root $-a/3$ if $c = a^3/27$, for any other value of $c$, cubic with $a^2 = 3b$ will have a single real root.
\end{proof}

\begin{figure}
\centering
\includegraphics[height=8.8cm, width=\textwidth]{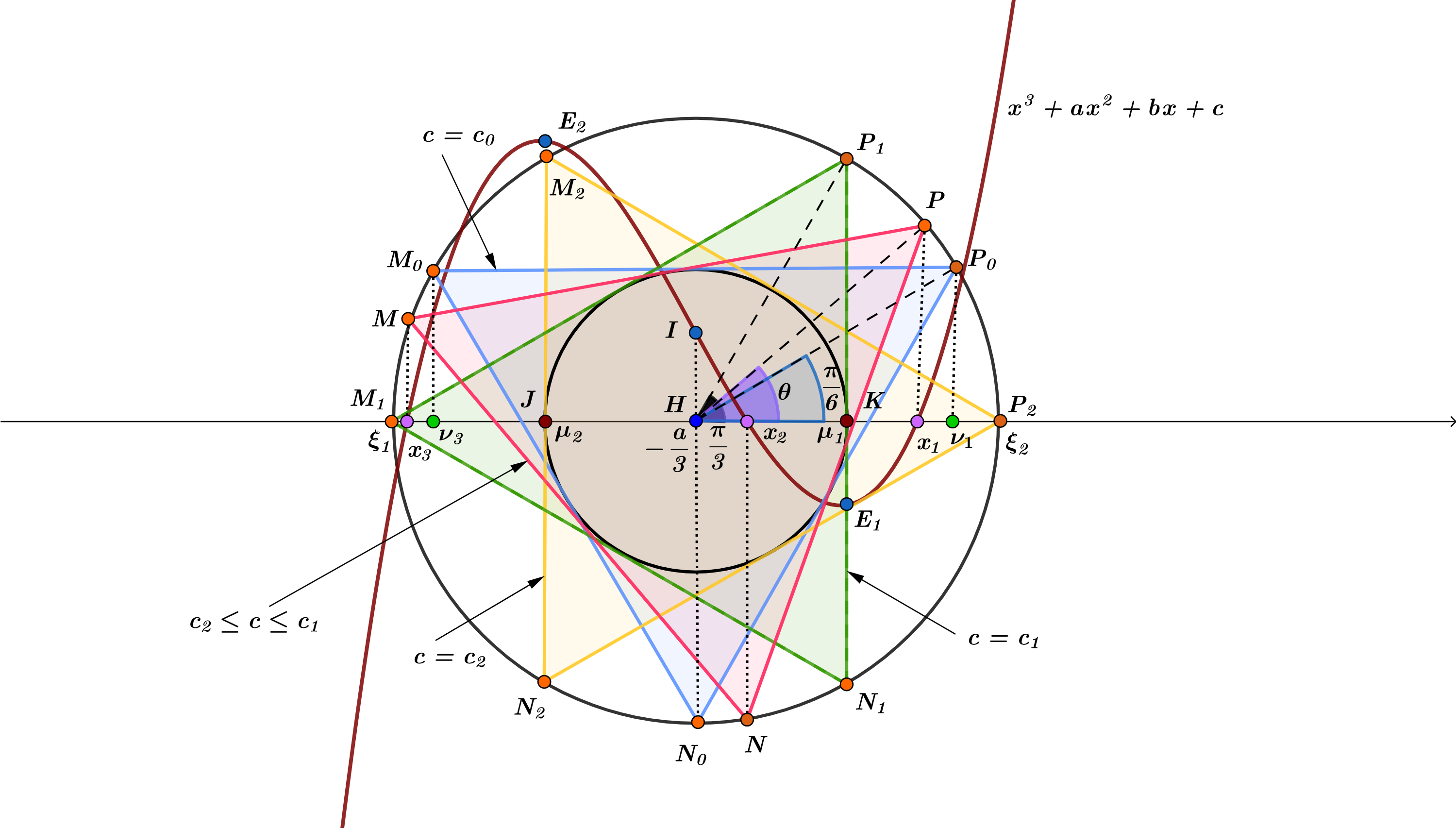}}
{\begin{minipage}{36em}
\scriptsize
\vskip.3cm
\begin{center}
{\bf Figure 1} \\
\vskip.3cm
{\bf The Siebeck--Marden--Northshield triangle and its rotation upon variation of $\bm{c}$}
\end{center}
\end{minipage}
\end{figure}

\begin{theorem}
The maximum length of the interval containing the three real roots of the cubic $x^3 + a x^2 + b x + c$ is $\sqrt{12} r = (\sqrt{12}/3) \sqrt{a^2 - 3b} = l$ and this is achieved when $a^2 - 3b > 0$ and $c = c_0 \in [c_2, c_1]$. The minimum length of this interval is $3r = \sqrt{a^2 - 3b}$, occurring when  $a^2 - 3b > 0$ and $c = c_{1,2}$ \cite{27, 29}.
\end{theorem}

\begin{proof}
The roots of the cubic $x^3 + a x^2 + b x + c\,,\ $ with $a^2 - 3b > 0$ and $c = c_0 = - 2a^3/27 + ab/3$ are $\nu_{1,3} = -a/3 \pm l/2 = -a/3 \pm (\sqrt{3}/3) \sqrt{a^2 - 3 b}$ and $\nu_2 = - a/3$. Note that $\nu_{1,3}$ are symmetric with respect to the centroid $-a/3$ (point $H$ on Figure 1) of the Siebeck--Marden--Northshield triangle onto which the middle root projects. The $y$-coordinates of the vertices of the Siebeck--Marden--Northshield triangle, which project onto the biggest and smallest roots, are both equal to $(\sqrt{3}/3) \sqrt{a^2 - 3 b}$, that is, the triangle has a side parallel to the abscissa. In Figure 1, this is the triangle $M_0 N_0 P_0$ with $M_0 P_0$ parallel to the $x$-axis. The centroid of the Siebeck--Marden--Northshield triangle is at point $H$ and the angle which $P_0 H$ forms with the abscissa is $\pi/6$. The distance between the roots $\nu_3$ and $\nu_1$ is exactly equal to $l = \sqrt{12} r = (\sqrt{12}/3) \sqrt{a^2 - 3b}$. \\
The roots of the cubic $x^3 + a x^2 + b x + c\,,\ $ with $a^2 - 3b > 0$ and $c = c_{1,2}$ are the double root $\mu_{1,2} = -a/3 \pm r = -a/3 \pm (1/3) \sqrt{a^2 - 3 b}$ (points $K/J$ on Figure 1), which are roots to $p_3'(x) = 3 x^2 + 2 a x + b = 0$, and the single root $\xi_{1,2} =  - a - 2 \mu_{1,2} = - a/3 \mp 2r = -a/3 \mp (2/3) \sqrt{a^2 - 3 b}$ (points $M_1/P_2$). \\
In these cases, the respective equilateral triangles have a side perpendicular to the abscissa: these are $M_1 N_1 P_1$ with $P_1 N_1 \,\, \bot \,\, Ox$, when $c = c_1$, and $M_2 N_2 P_2$ with $M_2 N_2 \,\, \bot \,\, Ox$, when $c = c_2$ --- see Figure 1. This yields the shortest possible interval, containing the three real roots of the cubic (one of which is repeated). It is equal to the height of Siebeck--Marden--Northshield triangle or $3r = \sqrt{a^2 - 3b}$. The angle which $P_1 H$ forms with the abscissa is $\pi/3$ and the angle which $P_2 H$ forms with the abscissa is $0$ (as $P_2$ is on the $x$-axis). \\
For any other $c$ such that $c_2 \le c \le c_1$, the three real roots of the cubic lie within an interval of length between $3r$ and $\sqrt{12} r \approx 3.4641r$. In Figure 1, these roots are represented by $MNP$ --- the vertices of the general Siebeck--Marden--Northshield triangle. The angle $\theta$, which $PH$ forms with the $x$-axis, satisfies $0 \le \theta \le \pi/3$. Note that $c = c_2$ corresponds to $\theta = 0$; $c = c_0$ corresponds to $\theta = \pi/6$; and $c = c_1$ corresponds to $\theta = \pi/3$.
\end{proof}

\subsection{Roles Played by the Coefficients of the Cubic}

\n
The coefficient $a$ of the quadratic term of $x^3 + a x^2 + b x + c$ selects the centroid of the Siebeck--Marden--Northshield triangle. This is where the inflection point of $x^3 + a x^2 + b x + c$ projects onto the abscissa. For any given $a$, the coefficients $b < a^2/3$ of the linear term of $x^3 + a x^2 + b x + c$ determines the radius $r = (1/3) \sqrt{a^2 - 3b}$ of the inscribed circle (the triangle exists only for $a^2 - 3b > 0$) and, hence, the size of Siebeck--Marden--Northshield triangle. The inscribed circle projects to an interval on the abscissa with endpoints equal to the projections of the critical points of the cubic, the distance between which is always $2 r = (2/3) \sqrt{a^2 - 3b}$ --- the radius of the circumscribed circle. The inflection point of the cubic is always in the middle between the two critical points. For $a^2 - 3b > 0$, the variation of the free term $c$ of $x^3 + a x^2 + b x + c$ rotates the Siebeck--Marden--Northshield triangle counterclockwise from the position of $M_2 N_2 P_2$, when $c = c_2$ ($\theta = 0$), through that of $M_0 N_0 P_0$, when $c = c_0$ ($\theta = \pi/6$), up to that of $M_1 N_1 P_1$, when $c = c_1$ ($\theta = \pi/3$).

\begin{theorem}
If the cubic $x^3 + a x^2 + b x + c$ has three real roots, then no root is greater than $-a/3 + 2r = -a/3 + (2/3) \sqrt{a^2 - 3b}$ and no root is smaller than $-a/3 - 2r = -a/3 - (2/3) \sqrt{a^2 - 3b}$.
\end{theorem}

\begin{proof}
Three real roots of the cubic exist if $a^2 - 3b > 0$ and $c_2 \le  c \le c_1$. The $x$-coordinate projections of the three real roots are between $\xi_1$ and $\xi_2$ --- points $M_1$ and point $P_2$, respectively, in Figure 1. When $c = c_2$, the biggest of the three roots is given by $\xi_2 =  - a/3 + 2r = -a/3 + (2/3) \sqrt{a^2 - 3 b}$ (point $P_2$ on Figure 1). The other two roots in this case are both equal to $\mu_2 = - a/3 - (1/3) \sqrt{a^2 - 3 b}$ (point $J$). Increasing $c$ from $c_2$ towards $c_1$ results in the rotation of the Siebeck--Marden--Northshield triangle around its centroid and, for any $c$ such that $c_2 \le c \le c_1$, the biggest of the three real roots will have $x$ coordinate projection between $\mu_1$ and $\xi_2$ (points $K$ and $P_2$, respectively). Hence, no root can lie further than $-a/3 + (2/3) \sqrt{a^2 - 3b}$. \\
Similarly, the smallest of the three real roots is between $\xi_1 =  - a/3 - 2r = -a/3 - (2/3) \sqrt{a^2 - 3 b}$  and $\mu_2 = - a/3 - (1/3) \sqrt{a^2 - 3 b}$ --- points $M_1$ and $J$, respectively. When $c = c_1$, the smallest root is exactly equal to $\xi_1 = -a/3 - (2/3) \sqrt{a^2 - 3 b}$ and the other two roots are both equal to $\mu_1 = - a/3 + (1/3) \sqrt{a^2 - 3 b}$. For any $c$ such that $c_2 \le c \le c_1$, the smallest of the three real roots will have $x$ coordinate projection between $\xi_1$ and $\mu_2$ (points $M_1$ and $J$, respectively). Hence, no root can be smaller than $-a/3 - (2/3) \sqrt{a^2 - 3b}$.
\end{proof}

\subsection{Relationship to the Vi\`ete Trigonometric Formul\ae \, for the Roots of the Cubic}

\n
Under the coordinate translation $x \to x + a/3$, the centroid of the Siebeck--Marden--Northshield triangle gets at the origin of the new coordinate system (this also depresses the cubic). It is straightforward to determine the $x$-coordinate projections of the vertices of the triangle in this new coordinate system and hence find that the {\it real} roots of the cubic are given by:
\b
\label{n1}
x_1 \!\!\! & = & \!\!\!  -\frac{a}{3} \, + \, \frac{2}{3} \sqrt{a^2 - 3b} \, \cos \theta, \\
\label{n2}
x_2 \!\!\! & = & \!\!\!  -\frac{a}{3} \, - \, \frac{2}{3} \sqrt{a^2 - 3b} \, \cos\!\left(\theta + \frac{\pi}{3}\right) , \\
\label{n3}
x_3 \!\!\! & = & \!\!\!  -\frac{a}{3} \, - \, \frac{2}{3} \sqrt{a^2 - 3b} \, \cos\!\left(\theta - \frac{\pi}{3} \right).
\e
These are the Vi\`ete trigonometric formul\ae \, for the {\it real} real roots of the cubic. \\
One can verify that the Vi\`ete formul\ae\, $x_1 + x_2 + x_3 = -a$ and $x_1 x_2 + x_2 x_3 + x_1 x_3 = b$ are immediately satisfied. Due to the third Vi\`ete formula, the product of the three roots is $-c$. Thus one obtains
\b
\label{tse}
c = - \frac{2}{27} a^3 + \frac{1}{3} ab - \cos (3\theta) \, \frac{2}{27} \sqrt{(a^2 - 3b)^3}
\e
and from this, one easily gets that the angle $\theta$ of rotation of the Siebeck--Marden--Northshield triangle around its centroid is
\b
\label{theta}
\theta = \frac{1}{3} \arccos \!\left[ -\frac{2a^3 - 9ab + 27c}{2\sqrt{(a^2 - 3b)^3}} \right] \!.
\e
Note that when $\theta = 0$, formula (\ref{tse}) yields $c = c_2$; when $\theta = \pi/6$, one gets $c = c_0$; and when $\theta = \pi/3$, from (\ref{tse}) one gets $c = c_1$. That is, the expression in the square brackets above takes continuous values from $-1$ (when $c = c_1$), through $0$ (when $c = c_0$), to $1$ (when $c = c_2$). \\
When $a^2 - 3b > 0$, but $c \notin [c_2, c_1]$, the value of the expression in the square brackets of (\ref{theta}) has absolute value greater than 1. In this case, $\theta$ is a complex number (purely imaginary, if $c < c_2$). However, (\ref{n1})--(\ref{n3}) still produce the roots of the cubic --- with only one of them real and the other two --- complex-conjugates of each other. \\
When $a^2 - 3b < 0$, there will again be only one real root and a pair of complex-conjugate roots. However, formul\ae \, (\ref{n1})--(\ref{n3}) no longer work. $\theta$ is again a complex number. One needs to get the roots, symmetric to those in (\ref{n1})--(\ref{n3}) with respect to the $y$-axis in the $x \to x + a/3$ coordinate system:
\b
\label{n4}
\widehat{x}_1 \!\!\! & = & \!\!\!  -\frac{a}{3} \, - \, \frac{2}{3} \sqrt{a^2 - 3b} \, \cos \theta, \\
\label{n5}
\widehat{x}_2 \!\!\! & = & \!\!\!  -\frac{a}{3} \, + \, \frac{2}{3} \sqrt{a^2 - 3b} \, \cos\!\left(\theta + \frac{\pi}{3}\right) , \\
\label{n6}
\widehat{x}_3 \!\!\! & = & \!\!\!  -\frac{a}{3} \, + \, \frac{2}{3} \sqrt{a^2 - 3b} \, \cos\!\left(\theta - \frac{\pi}{3} \right).
\e
In this case, one has $\widehat{x}_1 + \widehat{x}_2 + \widehat{x}_3 = -a$ and $\, \widehat{x}_1 \widehat{x}_2 + \widehat{x}_2 \widehat{x}_3 + \widehat{x}_1 \widehat{x}_3 = b$ again, while
\b
- \widehat{x}_1 \widehat{x}_2 \widehat{x}_3 = c = - \frac{2}{27} a^3 + \frac{1}{3} ab + \cos (3\theta) \, \frac{2}{27} \sqrt{(a^2 - 3b)^3}.
\e

\subsection{Isolation Intervals of the Roots of the Cubic}

\n
From the Siebeck--Marden--Northshield triangle, one can immediately read geometrically the isolation intervals of the three real roots of the cubic. These are as follows \cite{27, 29}:
\begin{itemize}

\item [(i)] For $c_2 \le c \le c_0: \quad$ $\nu_3 \le x_3 \le \mu_2, \quad \mu_2 \le x_2 \le -a/3, \quad$ and  $\quad \nu_1 \le x_1 \le \xi_2$.

\item [(ii)] For $c_0 \le c \le c_1: \quad$ $\xi_1 \le x_3 \le \nu_3, \quad -a/3 \le x_2 \le \mu_1, \quad$ and  $\quad \mu_1 \le x_1 \le \nu_1$.

\end{itemize}

\section{Nature of the Roots of the Quartic Equation}

It has been long since the complete root classification for the quartic equation was established --- in 1861, Cayley \cite{cayley} used the Sturmian constants (the leading coefficients in the Sturm functions modulo positive multiplicative factors) for the quartic (in binomial form) to achieve this. In this work, Cayley also proposed the Complete Root Classification for the cubic equation. \\
For the monic quartic polynomial
\b
p_4(x) = x^4 + a x^3 + b x^2 + c x + d,
\e
the sequence of Sturm functions for $p_4(x)$ are \cite{sturm}: $\pi_0(x) = p_4(x), \, \pi_1(x) =  p_4'(x), \, \pi_i(x) = $ rem$[\pi_{i-2}(x)/\pi_{i-1}(x)]$ ($i = 2, 3, 4$) --- the remainder of the division of the polynomial $\pi_{i-2}(x)$ by the polynomial $\pi_{i-1}(x)$. \\
The Sturmian constants for $p_4(x)$ are: $S_0 = 1$, $\, S_1 = 1$, $\, S_3 = 3a^2 - 8b$, $\, S_4 = -3 a^3 c + (b^2 - 6 d) a^2 + 14 a b c - 4 b^3 + 16 b d - 18 c^2$, and $S_5 = \Delta$ --- the discriminant of the quartic:
\b
\label{disc}
\Delta & \!\!=\!\! & 256 \, d^3 \,\, + \,\, (- 27 a^4 + 144 a^2 b - 192 a c - 128 b^2) \, d^2 \nonumber \\
& & + \,\, 2 \,(9 a^3 b c - 2 a^2 b^3 - 3 a^2 c^2 - 40 a b^2 c + 8 b^4 + 72 b c^2) \, d
\nonumber \\ & & - \,\, 4 a^3 c^3 + a^2 b^2 c^2 + 18 a b c^3 - 4 b^3 c^2 - 27 c^4.
\e
The number of real roots of $p_4(x)$ is equal to the excess of the number of variations of sign in the sequence of $\pi_j(x)$ at $- \infty$ over the number of variations of sign of these at $+ \infty$ (with vanishing terms ignored). Denoting in round brackets the sequence of the signs of the Sturm functions $\pi_j(x)$ at $+ \infty$ and in square brackets --- those at $- \infty$, one has \cite{cayley}:
\begin{itemize}
\item [(i)] $(+++++),[+-+-+]$ --- four real roots.
\item [(ii)] $(++-++),[+---+]$ --- no real roots.
\item [(iii)] $(+++-+),[+-+++]$ --- no real roots.
\item [(iv)] $(++--+),[+--++]$ --- no real roots.
\item [(v)] $(++++-),[+-+--]$ --- two real roots.
\item [(vi)] $(++-+-),[+----]$ --- cannot occur (impossible to have $S_3 < 0, \,\, S_4 > 0, $ and $S_5 = \Delta < 0$).
\item [(vii)] $(+++--),[+-++-]$ --- two real roots.
\item [(viii)] $(++---),[+--+-]$ --- two real roots.
\end{itemize}
It is clear from this complete root classification, that a positive discriminant of the quartic is associated with either four real roots or no real roots, while a negative one --- with two real roots. \\
In 1908, Petrovich \cite{pet} proposed a very interesting approach for finding a necessary and sufficient condition for the reality of all roots of a polynomial with real coefficients, $a_0 + a_1 z + \cdots + a_p z^p$, together with the roots of any polynomial composed of the set of any number of its first terms, i.e.
$a_0 + a_1 z + \cdots + a_n z^n, \,\, n \le p.$ \\
The Sturm method allows the immediate determination of the nature of the roots of a polynomial as, from the given coefficients, one can immediately determine the Sturm functions and hence establish their sequences of signs at $+ \infty$ and $- \infty$. \\
However, one would pose the following, in some sense, inverse, problem: would it be possible to ``{\it reverse-engineer}" a quartic, with {\it a priori} specified nature of its roots, by working at the level of the individual coefficients of the polynomial, that is, by knowing how to choose the coefficients of the quartic suitably [and not by merely expanding $(x - x_1)(x - x_2)(x - x_3)(x - x_4)$ in which the roots $x_i$ are chosen]? In this work, an alternative method is proposed with which not only complete root classification of the quartic is made (in the following Section), but also allows one to ``{\it reverse-engineer}" quartics. For example, could a quartic with $a = 5$ have four real roots? (The answer is yes.) And, if $a = 5$, what should $b, c, $ and $d$ sequentially be so that the quartic will have four real roots?  Indeed, a cubic (with guaranteed three real roots in this case) has to be solved in the process in order to achieve this, but the Vi\`ete trigonometric formul\ae \, for the roots of the cubic are quite easy to apply (as in the previous Section). \\
%\vskip.5cm
%\n
The discriminant of a monic polynomial of degree $n$ is $\prod_{i<j} (x_i - x_j)^2$, where $x_i$ are all $n$ roots (real or complex) of the polynomial. \\
If the discriminant is positive, then, obviously, the number of complex roots of the polynomial is a multiple of 4. If the discriminant is negative, then there are $2m + 1$ pairs of complex-conjugate roots, where $m \le (n - 2)/4.$ \\
If the quartic $p_4(x)$ has a positive discriminant $\Delta$, there are either four real roots or there are two pairs of complex-conjugate roots. On the other hand, if $\Delta$ is negative, then the quartic has two real roots and a pair of complex-conjugate roots. Vanishing discriminant means a repeated real root (of multiplicity 2, 3, or 4) or a repeated complex root (of multiplicity 2). One can find criteria to determine which of these three cases (zero, two, or four real roots) applies by deriving conditions on the coefficients of the quartic, rather than studying the discriminant $\Delta$ alone. \\
The polynomial $p_4(x)$ can be viewed as an element of congruence of quartic polynomials $\{ x^4 + a x^3 + b x^2 + c x + D \,\, | \,\, D \in \mbox{I} \hskip-2pt \mbox{R} \}$, the graphs of which foliate the $xy$-plane by the continuous variation of the foliation parameter $D$ --- the free term. All polynomials in the congruence have the same set of stationary points $M_i$ --- the roots of the polynomial $p_4'(x) = 4 x^3 + 3 a x^2 + 2 b x + c$. The stationary points of $p_4(x)$ are either a single global minimum [when $p_4'(x)$ has a single real root or a treble real root], or a local maximum, sandwiched by two local minima [when $p_4'(x)$ has three distinct real roots], or a saddle and a global minimum [when $p_4'(x)$ has a double real root and a single real root] --- as these are the determined by finding the roots of a cubic. Quartic polynomials with one stationary point can have either no real roots or two real roots, while the rest can have zero, two, or four real roots. Within this congruence of quartic polynomials, there are either one or three (not necessarily all different in the latter case) polynomials
\b
P_4^{(i)}(x) = p_4(x) - p_4(M_i),
\e
for which the abscissa is tangent to their graph at the stationary point $M_i$. Each polynomial $P_4^{(i)}(x)$ has a zero discriminant as each one of them has (at least one) at least double real root. These polynomials are shown in Figure 2 (for the case of polynomials with three distinct stationary points) and in Figure 3 (for those with a single stationary point).

\begin{center}
\begin{tabular}{cc}
\includegraphics[width=68mm]{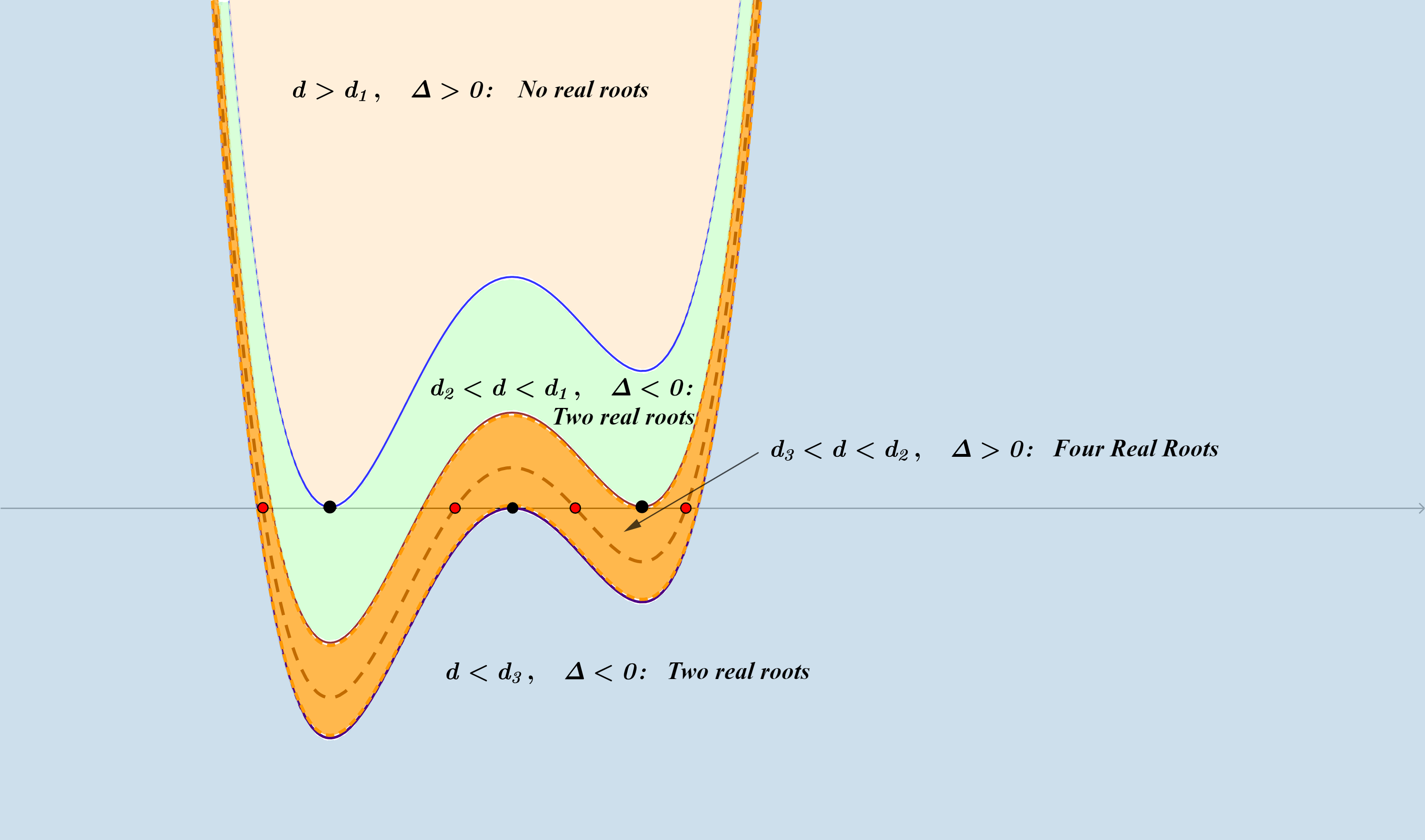} & \includegraphics[width=68mm,height=40mm]{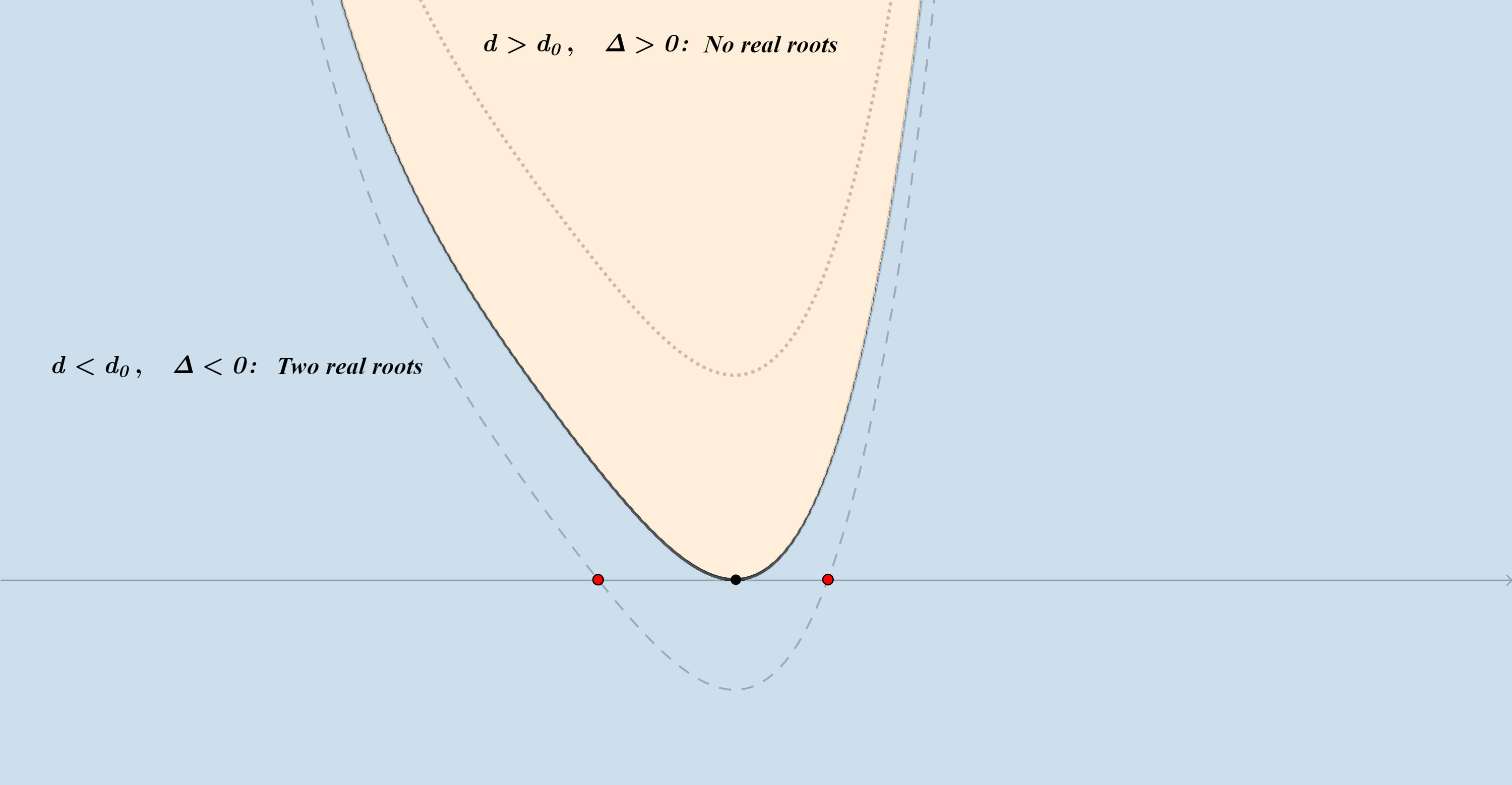} \\
& \\
{\scriptsize {\bf Figure 2}} &  {\scriptsize {\bf Figure 3}} \\
& \\
\multicolumn{1}{c}{\begin{minipage}{17em}
\scriptsize
\begin{center}
\vskip-.3cm
{\bf Quartic polynomials with three stationary points}
\end{center}
\end{minipage}}
& \multicolumn{1}{c}{\begin{minipage}{17em}
\scriptsize
\begin{center}
\vskip-.3cm
{\bf Quartic polynomials with a single stationary point}
\end{center}
\end{minipage}}
\\
\end{tabular}
\end{center}

\n
When the free term $d$ of $p_4(x)$ is large enough, the quartic will have no real roots and, for as long as the free term $d$ satisfies $d > p_4(M_i)$ for all $i$, the graph of this polynomial will lie entirely in the upper half-plane and there will be no real roots --- these are the polynomials in the band above the uppermost solid curve on Figure 2 and those above the single solid curve on Figure 3. Hence the band of polynomials whose graphs lie above the (uppermost) solid curve in Figures 2 and 3 all have positive discriminants. \\
If the free term $d$ of $p_4(x)$ is less than the smallest of $p_4(M_i)$, but larger than the other two values of $p_4(M_i)$, if these exist, then the graph of this quartic will cross the abscissa twice. Hence, all polynomials in this band (bounded by the two uppermost solid curves in Figure 2) will have two real roots and a negative discriminant. One should note however that in the case of a quartic polynomial with two double real roots (not shown in Figure 2), when $d$ is less than the two equal smallest values of $p_4(M_i)$ (at the two minima), but greater than the value of $p_4(x)$ at the maximum, there will be four real roots: the discriminant will not become negative, but will ``bounce back" from zero. \\
If the free term $d$ is less than $p_4(M_i)$ at the two minima, but greater than the value of $p_4(x)$ at the maximum, there will be four real roots and all polynomials in this band (bounded by the two lowermost solid curves on Figure 2) will be with positive discriminants. \\
Finally, if $d$ is smaller than all $p_4(M_i)$, then the polynomial will have two real roots and a negative discriminant --- the band below the lowermost curve on Figure 2. \\
Clearly, with the variation of the free term, the discriminant of the quartic changes sign three times (in case of three distinct stationary points) or once (in case of either a single global minimum or a saddle and a global minimum). This is indicative of the presence of a polynomial cubic in $d$ and, indeed, the discriminant $\Delta$ of the quartic [given in (\ref{disc})] is such polynomial [but it is quartic with respect to any other coefficient of $p_4(x)$]. \\
Consider the discriminant $\Delta_3$ of this cubic in $d$. It has a very simple form:
\b
\label{del3}
\Delta_3 = -314928 \, (a^3 - 4 a b + 8 c)^2 \, \left[ 4 c^2 + a (a^2 - 4b) c - \frac{1}{3} \, b^2 \left( a^2 - \frac{32}{9} b \right) \right]^3.
\e
The sign of the discriminant $\Delta_3$ and hence, the number of roots of the cubic in $d$ equation $\Delta(d) = 0$, depend on the sign of the term in the square brackets. The discriminant $\Delta_3 $ could also be zero [leading to a multiple root of $\Delta(d) = 0$] --- either by virtue of $c = - a^3/8 + ab/2$ or due to the term in the square brackets. \\
The term in the square brackets is quadratic in $c$. The discriminant of this quadratic is also very simple:
\b
\Delta_2 = \frac{1}{27} \, (3 a^2 - 8 b)^3
\e
and $\Delta_2$ is positive for $3 a^2 - 8b > 0$. In such case, the term in the square brackets of $\Delta_3$ is negative when $c$ is between the roots $C_{1,2}$ of the quadratic $4 c^2 + a (a^2 - 4b) c - (b^2/3)\,(a^2 - 32b/9)$, where
\b
C_{1,2} =  C_0 \,\, \pm \,\, \frac{\sqrt{3}}{72} \, \sqrt{(3 a^2 - 8 b)^3}
\e
and
\b
C_0 = - \frac{1}{8} \, a^3 + \frac{1}{2} \, ab.
\e
The discriminant $\Delta_3$, given by (\ref{del3}),  can be written conveniently as
\b
\Delta_3 = - 10077696 \, (c - C_0)^2 \, [(c - C_1)(c - C_2)]^3.
\e
Should $\Delta_3$ happen to be zero, either by virtue of $c = C_0$ or due to $c = C_{1,2}$ (the latter being real only when $b \le 3a^2/8$), then the cubic equation $\Delta(d) = 0$ will have a repeated root, that is, either a double real root $d^\dagger$ and a single real root $\widetilde{d}$, or a triple real root $\widetilde{d}$. In either case, the discriminant $\Delta(d)$ of the quartic will change sign only once --- at $\widetilde{d}$. \\
In the case of a triple root $\widetilde{d}$, the quartic will have the form $(x - x_0)^4$ with $x_0$ being a quadruple root. \\
If $c \ne C_0$ and $c \ne C_{1,2}$, the discriminant $\Delta(d)$ of the quartic will not be zero. It will have either 3 distinct roots $d_1 > d_2 > d_3$ (in the case of three stationary points of the quartic) or a single real root $d_0$ (in the case of a single stationary point (minimum) of the quartic). \\
In the case of three distinct real roots $d_1, \, d_2,$ and $d_3$ or a single real root $d_0$ of the discriminant $\Delta(d)$, these can be easily found using the Vi\`ete trigonometric formil\ae \, (\ref{n1})--(\ref{n3}):
\b
\label{d1}
d_1 & = & - \frac{A}{3} \, + \, \frac{2}{3} \sqrt{A^2 - 3B} \, \cos \Theta, \\
\label{d2}
d_2 & = & - \frac{A}{3} \, - \, \frac{2}{3} \sqrt{A^2 - 3B} \, \cos\!\left(\Theta + \frac{\pi}{3}\right), \\
\label{d3}
d_3 & = & - \frac{A}{3} \, - \, \frac{2}{3} \sqrt{A^2 - 3B} \, \cos\!\left(\Theta - \frac{\pi}{3}\right),
\e
where:
\b
A & = & \frac{1}{2} \left( -\frac{27}{128} \, a^4 + \frac{9}{8} \, a^2 b - \frac{3}{2} \, a c -  b^2 \right), \\
B & = & \frac{1}{16} \left( \frac{9}{8} \, a^3 b c - \frac{1}{4} \, a^2 b^3 - \frac{3}{8} \, a^2 c^2 - 5 a \,b^2 c + b^4+ 9 b \,c^2 \right), \\
C & = & \frac{1}{64} \left( - a^3 c^3 + \frac{1}{4} \, a^2 b^2 c^2 + \frac{9}{2} \, a b c^3 -  b^3 c^2 - \frac{27}{4} \, c^4 \right),
\e
and
\b
& & \hskip-0.8cm \theta = \frac{1}{3} \, \arccos \left[ -\frac{2 A^3 - 9 A B + 27 C}{2 \sqrt{(A^2 - 3B)^3}} \right] \nonumber \\ \nonumber \\
& & \hskip-0.8cm = \!\frac{1}{3} \!\arccos \!\left[ \!\sqrt{\frac{(3 a^2 - 8 b)^{-3}}{(243 a^6 - 1944 a^4 b + 3456 a^3 c + 4032 a^2 b^2 - 13824 a b c - 512 b^3 + 13824 c^2)^3}} \right. \nonumber \\ \nonumber \\
& & \hskip+1.5cm \times \Bigl( 19683 a^{12} - 314928 a^{10} b + 419904 a^9 c + 1959552 a^8 b^2 - 5038848 a^7 b c \nonumber \\ \nonumber \\
& & \hskip+2.1cm - \, 5847552 a^6 b^3 + 3172608 a^6 c^2 + 21399552 a^5 b^2 c + 8128512 a^4 b^4 \nonumber \\ \nonumber \\
& & \hskip+2.1cm - \, 25380864 a^4 b c^2 - 36274176 a^3 b^3 c - 3833856 a^2 b^5 + 11943936 a^3 c^3  \nonumber \\ \nonumber \\
& & \hskip+2.1cm + \, 55738368 a^2 b^2 c^2 + 17694720 a b^4 c - 262144 b^6 - 47775744 a b c^3  \nonumber \\ \nonumber \\
& & \hskip+2.1cm - \, 17694720 b^3 c^2 + 23887872 c^4  \Bigr) \biggr],
\e
provided that $A^2 - 3B = (1/65536) \, (3a^2 - 8b) \, (243 a^6 - 1944 a^4 b + 3456 a^3 c + 4032 a^2 b^2 - 13824 a b c - 512 b^3 + 13824 c^2)$ is positive.  The expression in the last pair of parenthesis is quadratic in $c$ and its discriminant is $-55296 (3 a^2 - 8 b)^3$. Hence, if $3 a^2 - 8 b > 0$, then this quadratic in $c$ is positive. This, in turn, leads to $A^2 - 3B > 0$. That is, $3 a^2 - 8 b > 0$ is a sufficient condition for $A^2 - 3B > 0$. Depending on $C$ --- whether it lies in the closed interval between the roots $-2A^3/27 +AB/3 \pm (2/27) \sqrt{(A^2 - 3B)^3}$ or not, see (\ref{c12}), there will be either three distinct real roots $d_1 > d_2 > d_3$ or a single real root (denoted by $d_0$), respectively. \\
However, if $A^2 - 3B < 0$, then, in order to determine the roots of the cubic equation $\Delta(d) = 0$, one should use the Vi\`ete formul\ae \, (\ref{d1})--(\ref{d3}) with the signs in front of the radicals inverted --- as in (\ref{n4})--(\ref{n6}). Only one of these roots, denote it by $d_0$ again, will be real.

\section{Complete Root Classification for the Quartic}

\n
In the literature, the term {\it Root Classification} refers to the determination of all possible cases of the polynomial roots (real and complex) and it consists of a list of the root multiplicities. The {\it Complete Root Classification} of parametric polynomials consists of the {\it Root Classification}, together with the conditions which the equation coefficients should satisfy for each case of the root classification --- see, for example, \cite{crc}.  \\
With the results obtained in the preceding section, one can now present the {\it Complete Root Classification} of the quartic $x^4 + a x^3 + b x^2 + c x + d$. For any $a$, one has:

\begin{itemize}

\item [{\bf (i)}] \underline{$\bm{b > 8a^2/3}, \quad \bm{c \ne C_0}, \quad \bm{d > d_0}$ --- {\bf no real roots}}

\begin{itemize}
\item []
{\it In this case, $\Delta_3 < 0$. Thus, the discriminant $\Delta(d)$ of the quartic has only one real root $d_0$. The quartic has one stationary point and positive discriminant $\Delta$. There are no real roots. Figure 3 applies --- dotted curve.}
\end{itemize}

\item [{\bf (ii)}] \underline{$\bm{b > 8a^2/3}, \quad \bm{c \ne C_0}, \quad \bm{d = d_0}$ --- {\bf two equal real roots}}

\begin{itemize}
\item []
{\it In this case, one again has $\Delta_3 < 0$ and the discriminant $\Delta(d)$ of the quartic having only one real root $d_0$. The quartic has one stationary point and zero discriminant $\Delta$. There are two equal real roots, but not a quadruple root as $b \ne 8a^2/3$. Figure 3 applies --- solid curve.}
\end{itemize}

\item [{\bf (iii)}] \underline{$\bm{b > 8a^2/3}, \quad \bm{c \ne C_0}, \quad \bm{d < d_0}$ --- {\bf two distinct real roots}}
\begin{itemize}
\item []
{\it Again, one has $\Delta_3 < 0$ and the discriminant $\Delta(d)$ of the quartic having only one real root $d_0$. The quartic has one stationary point and negative discriminant $\Delta$. There are two distinct real roots. Figure 3 applies --- dashed curve.}
\end{itemize}

\item [{\bf (iv)}] \underline{$\bm{b > 8a^2/3}, \quad \,\, \bm{c = C_0}, \quad \, \bm{d > \widetilde{d}}\quad \,$ {\bf [where} $\bm{\widetilde{d}}$ {\bf is the single root of  $\bm{\Delta(d)}]$} {\bf --- no}} \linebreak \underline{{\bf real roots}}
\begin{itemize}
\item []
{\it When $c = C_0$, one has $\Delta_3 = 0$. Thus, the discriminant $\Delta(d)$ of the quartic either has a repeated real root $d^\dagger$ and a single real root $\widetilde{d}$ (it may also have a treble real root $\widetilde{d}$). The discriminant $\Delta(d)$ of the quartic changes sign only once (at $\widetilde{d}$) and, hence, the quartic has one stationary point and positive discriminant $\Delta$, except for $d = d^\dagger$, where the discriminant $\Delta$ is zero. There are no real roots (if $d = d^\dagger$, there are two double complex roots). Figure 3 applies --- dotted curve.}
\end{itemize}

\item [{\bf (v)}] \underline{$\bm{b > 8a^2/3}, \quad \bm{c = C_0}, \quad \bm{d = \widetilde{d}}\quad$ {\bf [where} $\bm{\widetilde{d}}$ {\bf is the single root of  $\bm{\Delta(d)}]$ --- two}} \linebreak \underline{{\bf equal real roots}}

\begin{itemize}
\item []
{\it One again has $\Delta_3 = 0$ as $c = C_0$. The quartic has zero discriminant $\Delta$. There are two equal real roots. Figure 3 applies --- solid curve.}
\end{itemize}

\item [{\bf (vi)}] \underline{$\bm{b > 8a^2/3}, \quad \bm{c = C_0}, \quad \bm{d < \widetilde{d}}\quad$ {\bf [where} $\bm{\widetilde{d}}$ {\bf is the single root of  $\bm{\Delta(d)}]$} {\bf --- two}} \linebreak \underline{{\bf distinct real roots}}

\begin{itemize}
\item []
{\it Again, as $c = C_0$, one has $\Delta_3 = 0$. The quartic now has negative discriminant $\Delta$. There are two distinct real roots. Figure 3 applies --- dashed curve.}
\end{itemize}

\item [{\bf (vii)}] \underline{$\bm{b = 8a^2/3}, \quad \bm{c \ne C_0 = a^3/16}, \quad \bm{d > d_0}$ --- {\bf no real roots}}

\begin{itemize}
\item []
{\it When $\,b = 8a^2/3$ and $c \ne C_0 = a^3/16\,$, one has $\,\Delta_3 = -19683(a^3 - 16c)^8/65536$ $ < 0$. Thus, the discriminant $\Delta(d)$ of the quartic has only one real root $d_0$. The quartic has one stationary point and positive discriminant $\Delta$. There are no real roots. Figure 3 applies --- dotted curve.}
\end{itemize}

\item [{\bf (viii)}] \underline{$\bm{b = 8a^2/3}, \quad \bm{c \ne C_0= a^3/16}, \quad \bm{d = d_0}$ --- {\bf two equal real roots}}

\begin{itemize}
\item []
{\it The situation is the same as the one in case (vii) with the only difference that this time the quartic has zero discriminant $\Delta$. There are two equal real roots. Figure 3 applies --- solid curve.}
\end{itemize}

\item [{\bf (ix)}] \underline{$\bm{b = 8a^2/3}, \quad \bm{c \ne C_0 = a^3/16}, \quad \bm{d < d_0}$ --- {\bf two distinct real roots}}

\begin{itemize}
\item []
{\it Same situation as in cases (vii) and (viii), with the difference that the quartic now has negative discriminant $\Delta$. There are two distinct real roots. Figure 3 applies --- dashed curve.}
\end{itemize}

\item [{\bf (x)}] \underline{$\bm{b = 8a^2/3}, \quad \bm{c = C_0 = a^3/16}, \quad \bm{d > d_0 = a^4/256}$ --- {\bf no real roots}}

\begin{itemize}
\item []
{\it When $b = 8a^2/3$ and $c = C_0 = a^3/16$, one has $\Delta_3 = 0$ and the discriminant $\Delta(d)$ of the quartic being a complete cube: $\Delta(d) = -(a^4 - 256d)^3/65536$, that is, there is a triple root $d_0 = a^4/256$. With $d > d_0$, the quartic has one stationary point and positive discriminant $\Delta$. There are no real roots. Figure 3 applies --- dotted curve.}
\end{itemize}

\item [{\bf (xi)}] \underline{$\bm{b = 8a^2/3}, \,\, \bm{c = C_0 = a^3/16}, \,\, \bm{d = d_0 = a^4/256}$ --- {\bf four equal real roots} $\bm{-a/4}$}

\begin{itemize}
\item []
{\it The quartic in this case is $(x + a/4)^4$. It has one stationary point, zero discriminant $\Delta$ and quadruple root $-a/4$. Figure 3 applies --- solid curve.}
\end{itemize}

\item [{\bf (xii)}] \underline{$\bm{b = 8a^2/3}, \quad \bm{c = C_0 = a^3/16}, \quad \bm{d < d_0 = a^4/256}$ --- {\bf two distinct real roots}}
\begin{itemize}
\item []
{\it Same situation as in cases (x) and (xi), with the difference that the quartic now has negative discriminant $\Delta$. There are two distinct real roots. Figure 3 applies --- dashed curve.}
\end{itemize}

\pagebreak

\item [{\bf (xiii)}] \underline{$\bm{b < 8a^2/3}, \quad \bm{C_2 < c < C_1}, \quad \bm{c \ne C_0}, \quad \bm{d > d_1}$ --- {\bf no real roots}}

\begin{itemize}
\item []
{\it In this case, one has $\Delta_3 > 0$ and the discriminant $\Delta(d)$ of the quartic having three distinct real roots $d_1 > d_2 > d_3$.With $d > d_1$, the quartic has three stationary points and positive discriminant $\Delta$. There are no real roots. Figure 2 applies --- curves in the band above the uppermost curve.}
\end{itemize}

\item [{\bf (xiv)}] \underline{$\bm{b < 8a^2/3}, \quad \bm{C_2 < c < C_1}, \quad \bm{c \ne C_0}, \quad \bm{d = d_1}$ --- {\bf two equal real roots}}

\begin{itemize}
\item []
{\it This time the discriminant $\Delta(d)$ of the quartic is zero. There are two equal real roots (the abscissa is tangent to the quartic at one of its minima). Figure 2 applies --- the uppermost curve.}
\end{itemize}

\item [{\bf (xv)}] \underline{$\bm{b < 8a^2/3}, \,\,\,\, \bm{C_2 < c < C_1}, \,\,\,\, \bm{c \ne C_0}, \,\,\,\, \bm{d_2 < d < d_1}$ --- {\bf two distinct real roots}}

\begin{itemize}
\item []
{\it This time the discriminant $\Delta(d)$ of the quartic is negative. There are two distinct real roots. Figure 2 applies --- curves between the two uppermost curves.}
\end{itemize}

\item [{\bf (xvi)}] \underline{$\bm{b < 8a^2/3}, \quad \,\, \bm{C_2 < c < C_1}, \quad \,\, \bm{c \ne C_0}, \qquad \bm{d = d_2}$ --- {\bf four real roots, two of}} \\
\underline{{\bf which (the two largest or the two smallest) are equal}}

\begin{itemize}
\item []
{\it The quartic has zero discriminant $\Delta$ again. There are four real roots, two of which are equal (the abscissa is tangent to the quartic at the other of its minima, see case xiii). Figure 2 applies --- the second curve from the top.}
\end{itemize}

\item [{\bf (xvii)}] \underline{$\bm{b < 8a^2/3}, \,\,\,\, \bm{C_2 < c < C_1}, \,\,\,\, \bm{c \ne C_0}, \,\,\,\,\bm{d_3 < d < d_2}$ --- {\bf four distinct real roots}}
\begin{itemize}
\item []
{\it The discriminant of the quartic is now positive. There are four distinct real roots. Figure 3 applies --- the dashed curve.}
\end{itemize}

\item [{\bf (xviii)}] \underline{$\bm{b < 8a^2/3}, \,\,\,\,\,\, \bm{C_2 < c < C_1}, \,\,\,\,\, \bm{c \ne C_0}, \,\,\,\,\, \bm{d = d_3}$ --- {\bf four real roots, the middle}} \\ \underline{{\bf two of which are equal}}

\begin{itemize}
\item []
{\it The quartic has zero discriminant $\Delta$ again. There are four real roots, with the middle two being equal (the abscissa is tangent to the quartic at its local maximum). Figure 2 applies --- the lowermost curve.}
\end{itemize}

\item [{\bf (xix)}] \underline{$\bm{b < 8a^2/3}, \,\,\,\, \bm{C_2 < c < C_1}, \,\,\,\, \bm{c \ne C_0}, \,\,\,\, \bm{d < d_3}$ --- {\bf two distinct real roots}}

\begin{itemize}
\item []
{\it The quartic has negative discriminant $\Delta$ again. There are two distinct real roots. Figure 2 applies --- curves below the lowermost curve.}
\end{itemize}

\pagebreak

\item [{\bf (xx)}] \underline{$\bm{b < 8a^2/3}, \,\,\,\,\, \bm{c = C_{1,2}}, \quad \bm{d > \widetilde{d}}\quad $ {\bf [where} $\bm{\widetilde{d}}$ {\bf is the single root of $\bm{\Delta(d)}${\bf ] --- no}}} \linebreak \underline{{\bf real roots}}
\begin{itemize}
\item []
{\it When $c = C_1$ or $c = C_2$, one has $\Delta_3 = 0$. The discriminant $\Delta(d)$ of the quartic has a double real root $d^\dagger$ and a single real root $\widetilde{d}$ (it may also have a treble real root $\widetilde{d}$). Thus it changes sign only once --- at $\widetilde{d}$. The stationary points of the quartic are a global minimum and a saddle (the latter to the right of the minimum if $c = C_1$ and to the left of it if $c = C_2$). When $d > \widetilde{d}$, the quartic has a positive discriminant $\Delta(d)$ and no real roots.}
\end{itemize}

\item [{\bf (xxi)}] \underline{$\bm{b < 8a^2/3}, \,\,\,\,\, \bm{c = C_{1,2}}, \,\,\,\,\, \bm{d = \widetilde{d}}\,\,\,\, $ {\bf [where} $\bm{\widetilde{d}}$ {\bf is the single root of $\bm{\Delta(d)}${\bf ] --- two}}} \linebreak \underline{{\bf equal real roots}}\begin{itemize}
\item []
{\it This time the quartic has a zero discriminant $\Delta(d)$ and there are two equal real roots. The abscissa is tangent to the graph at the double root --- the minimum. The saddle point is above the abscissa.}
\end{itemize}

\item [{\bf (xxii)}] \underline{$\bm{b < 8a^2/3}, \,\,\,\, \bm{c = C_{1,2}}, \,\,\, \bm{d^\dagger < d < \widetilde{d}}\,\,\, $ {\bf [where} $\bm{\widetilde{d}}$ {\bf is the single root of $\bm{\Delta(d)}${\bf \, and}}} \linebreak \underline{{\bf $\bm{d^\dagger}$ is the double root of $\bm{\Delta(d)}]$ --- two distinct real roots}}
\begin{itemize}
\item []
{\it The discriminant of the quartic is negative and there are two distinct real roots. The minimum is below the $x$-axis, while the saddle is above it.}
\end{itemize}

\item [{\bf (xxiii)}] \underline{$\bm{b < 8a^2/3}, \quad \bm{c = C_{1,2}}, \quad \bm{d = d^\dagger} \quad $ {\bf [where} $\bm{d^\dagger}$ {\bf is the double root of $\bm{\Delta(d)}${\bf ] ---}}} \linebreak \underline{{\bf four real roots, three of which equal}}
\begin{itemize}
\item []
{\it The discriminant of the quartic is zero. There are four real roots and three of them are equal. The quartic has a saddle point and it lies on the abscissa --- at the triple root. The minimum is below the $x$-axis. }
\end{itemize}

\item [{\bf (xxiv)}] \underline{$\bm{b < 8a^2/3}, \quad \bm{c = C_{1,2}}, \quad \bm{d < d^\dagger} \quad $ {\bf [where} $\bm{d^\dagger}$ {\bf is the double root of $\bm{\Delta(d)}${\bf ] ---}}} \linebreak \underline{{\bf two distinct real roots}}
\begin{itemize}
\item []
{\it The discriminant of the quartic ``bounces back" from zero and is again negative. There are two real roots. The saddle and the minimum of the quartic are both below the abscissa.}
\end{itemize}

\item [{\bf (xxv)}] \underline{$\bm{b < 8a^2/3}, \,\,\,\, \bm{c = C_0}, \,\,\,\, \bm{d > d^\dagger}\,\,\,\,\, $ {\bf [where} $\bm{d^\dagger}$ {\bf is the double root of $\bm{\Delta(d)}${\bf ] --- no}}} \linebreak \underline{{\bf real roots}}
    \begin{itemize}
\item []
{\it When $c = C_0$, one has $\Delta_3 = 0$. The discriminant $\Delta(d)$ of the quartic has a double real root $d^\dagger$ and a single real root $\widetilde{d}$ (it may also have a treble real root $\widetilde{d}$). The quartic has three stationary points, but its discriminant changes sign only once --- at $\widetilde{d} < d^\dagger$. The quartic is a symmetric curve with respect to its local maximum.  When $d > d^\dagger$, the quartic has a positive discriminant $\Delta(d)$ and no real roots.}
\end{itemize}

\item [{\bf (xxvi)}] \underline{$\bm{b < 8a^2/3}, \,\,\,\, \bm{c = C_0}, \,\,\,\, \bm{d = d^\dagger} \,\,\,$ {\bf [where} $\bm{d^\dagger}$ {\bf is the double root of $\bm{\Delta(d)}$}{\bf ] --- two}} \linebreak \underline{{\bf pairs of equal real roots}}
\begin{itemize}
\item []
{\it The discriminant of the quartic is zero. As the quartic``sinks" with $d$ becoming smaller, the two local minima ``land" simultaneously on the abscissa, that is, when $d = d^\dagger$, the $x$-axis is tangent to the graph of the quartic at both its minima. There are two pairs of equal real roots.}
\end{itemize}

\item [{\bf (xxvii)}] \underline{$\bm{b < 8a^2/3}, \,\,\, \bm{c = C_0}, \,\,\, \bm{\widetilde{d} < d < d^\dagger} \,\,\,$ {\bf [where} $\bm{d^\dagger}$ {\bf is the double root of $\bm{\Delta(d)}$}{\bf \, and}} \linebreak \underline{{\bf $\bm{\widetilde{d}}$ is the single root of $\bm{\Delta(d)}]$ --- four distinct real roots}}
\begin{itemize}
\item []
{\it The discriminant of the quartic ``bounces back" from zero and is again positive. The two local minima are now below the abscissa, while the local maximum is above it. There are four distinct real roots.}
\end{itemize}

\item [{\bf (xxviii)}] \underline{$\bm{b < 8a^2/3}, \quad \bm{c = C_0}, \quad \bm{d = \widetilde{d}} \,\,\,$ {\bf [where} $\bm{\widetilde{d}}$ {\bf is the single root of $\bm{\Delta(d)]}$}{\bf \, --- four}} \linebreak \underline{{\bf real roots, the middle two of which --- equal}}
\begin{itemize}
\item []
{\it The discriminant of the quartic is again zero. The abscissa is tangent to the quartic at its local maximum. There are four real roots and the two in the middle are equal --- at the local maximum.}
\end{itemize}

\item [{\bf (xxix)}] \underline{$\bm{b < 8a^2/3}, \quad \bm{c = C_0}, \quad \bm{d < \widetilde{d}} \,\,\,\,\,$ {\bf [where} $\bm{\widetilde{d}}$ {\bf is the single root of $\bm{\Delta(d)]}$}{\bf \, --- two}} \linebreak \underline{{\bf distinct real roots}}
\begin{itemize}
\item []
{\it The discriminant of the quartic is negative (it changes sign at $\widetilde{d}$). The local maximum now lies below the abscissa. There are two distinct real roots.}
\end{itemize}

\item [{\bf (xxx)}] \underline{$\bm{b < 8a^2/3}, \,\,\,\, \bm{c \notin [C_2, C_1]}, \,\,\,\, \bm{d > d_0}$ {\bf --- no real roots}}
\begin{itemize}
\item []
{\it When $c \notin [C_2, C_1]$, one has $\Delta_3 < 0$. The discriminant $\Delta(d)$ of the quartic has one real root $d_0$. Thus it changes sign only once. The quartic has one stationary point. When $d > d_0$, the quartic has a positive discriminant $\Delta(d)$ and no real roots. Figure 3 applies --- the dotted curve.}
\end{itemize}

\item [{\bf (xxxi)}] \underline{$\bm{b < 8a^2/3}, \,\,\,\, \bm{c \notin [C_2, C_1]}, \,\,\,\, \bm{d = d_0}$ {\bf --- two equal real roots}}
\begin{itemize}
\item []
{\it As  $c \notin [C_2, C_1]$, one has $\Delta_3 < 0$ again. This time the quartic has a zero discriminant $\Delta(d)$ and there are two equal real roots. Figure 3 applies --- the solid curve.}
\end{itemize}

\item [{\bf (xxxii)}] \underline{$\bm{b < 8a^2/3}, \,\,\,\, \bm{c \notin [C_2, C_1]}, \,\,\,\, \bm{d < d_0}$ {\bf --- two distinct real roots}}
\begin{itemize}
\item []
{\it Again, $\Delta_3 < 0$, as $c \notin [C_2, C_1]$. The discriminant $\Delta(d)$ of the quartic is negative and there are two distinct real roots. Figure 3 applies --- the dashed curve.}
\end{itemize}

\end{itemize}

\section{The Quartic Equation and its Tetrahedron}

\n
From the {\it Complete Root Classification} in the previous Section, one can immediately identify the necessary and sufficient conditions for the quartic to have four real roots. These are given, according to the root multiplicities, by the following four Theorems:
\begin{theorem}
The quartic polynomial $x^4 + a x^3 + b x^2 + c x + d$ has \underline{four distinct real roots} if, and only if, its coefficients satisfy all of the following three conditions:
\end{theorem}
\vskip-.75cm
\b
& & \hskip-1cm \bm{(i)} \quad \,\,\, b < \frac{3}{8} a^2, \\
& & \hskip-1cm \bm{(ii)} \quad C_2 < c < C_1, \\
& & \hskip-1cm \bm{(iii)} \,\,\, \left( c \ne  C_0 \,\,\, \mbox{and} \,\,\, d_3 < d < d_2 \right) \,\,\, \mbox{or} \,\,\, ( c = C_0 \,\,\, \mbox{and} \,\,\, \widetilde{d} < d < d^\dagger).
\e
$[\!${\it These are cases (xvii) and (xxvii), respectively, of the Complete Root Classification.}$]$

\begin{theorem}
The quartic polynomial $x^4 + a x^3 + b x^2 + c x + d$ has \underline{four real roots, exactly} \underline{two of which equal}, if, and only if, its coefficients satisfy any one of the following three conditions:
\end{theorem}
\vskip-.75cm
\b
& & b < \frac{3}{8} a^2, \quad C_2 < c < C_1,  \quad c \ne C_0,  \quad d = d_2, \\
& & b < \frac{3}{8} a^2, \quad C_2 < c < C_1, \quad c \ne C_0, \quad d = d_3, \\
& & b < \frac{3}{8} a^2, \quad c = C_0, \quad d = \widetilde{d}.
\e
$[\!${\it From the Complete Root Classification, these are the respective cases: (xvi) --- the largest two or the smallest two of the four roots equal, (xviii) --- the middle two of the four roots equal, and (xxviii) --- again, the middle two of the four roots equal.}$]$

\begin{theorem}
The quartic polynomial $x^4 + a x^3 + b x^2 + c x + d$ has \underline{two pairs of equal real} \underline{roots}, if, and only if, its coefficients satisfy:
\end{theorem}
\vskip-.75cm
\b
b < \frac{3}{8} a^2, \quad c = C_0,  \quad d = d^\dagger.
\e
$[\!${\it This is case (xxvi) of the Complete Root Classification.}$]$

\begin{theorem}
The quartic polynomial $x^4 + a x^3 + b x^2 + c x + d$ has \underline{four real roots, exactly} \underline{three of which equal}, if, and only if, its coefficients satisfy: \end{theorem}
\vskip-.75cm
\b
b < \frac{3}{8} a^2, \quad c = C_{1,2}, \quad d = d^\dagger.
\e
$[\!${\it This is case (xxiii) of the Complete Root Classification.}$]$

\vskip.7cm
\n
The quartic polynomial $p_4(x) = x^4 + a x^3 + b x^2 + c x + d$ with four real roots (at least two distinct) --- subject to any of {\bf Theorem 4} to {\bf Theorem 7} --- is associated with a regular tetrahedron and its symmetries \cite{north}, \cite{chalkley}, \cite{nickalls2}, \cite{nickalls1}. \\
The following Theorem is stated by Northshield in \cite{north}: {\it Given a quartic $p_4(x)$ with four real roots (at least two distinct), those roots are the first coordinate projections of a regular tetrahedron in I \hskip-6.5pt R$^3$. That tetrahedron has a unique inscribed sphere, which projects onto an interval whose endpoints are the two roots of $p''(x)$. The vertices of an equilateral triangle around that sphere project onto the roots of $p'(x)$}. \\
This equilateral triangle is exactly the Siebeck--Marden--Northshield triangle for the cubic polynomial $p_4'(x) = 4 x^3 + 3 a x^2 + 2 b x + c$. It lies in the $xy$-plane --- this is the triangle $M\!N\!P$ in Figure 4.

\begin{figure}
\centering
\includegraphics[height=9cm, width=\textwidth]{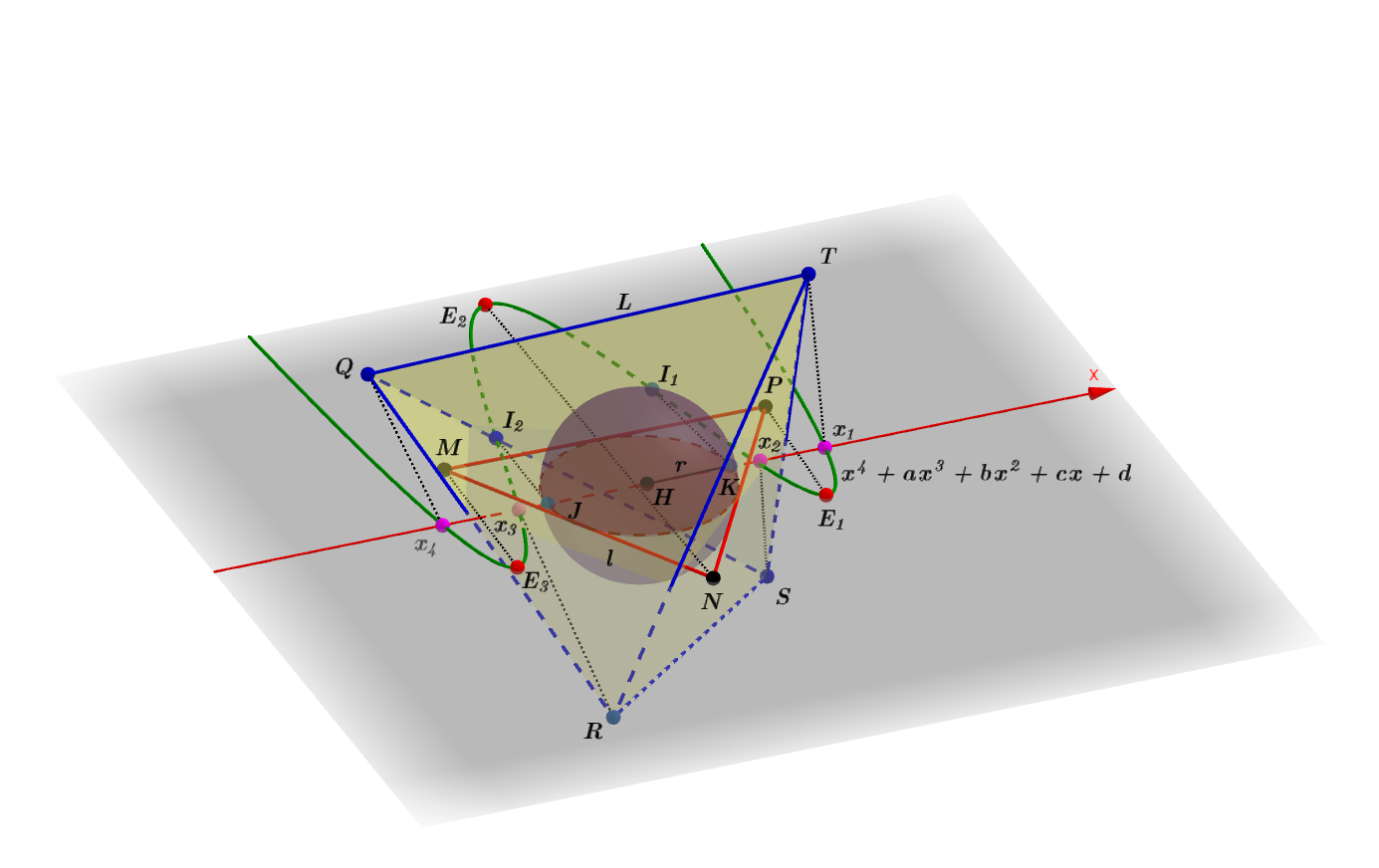}}
{\begin{minipage}{36em}
\scriptsize
\vskip.3cm
\begin{center}
{\bf Figure 4} \\
\vskip.3cm
{\bf The quartic polynomial and its regular tetrahedron}
\end{center}
\end{minipage}
\end{figure}

\n
Clearly, the Siebeck--Marden--Northshield triangle exists when $b < 3a^2/8$ and $C_2 \le c  \le C_1$. \\
In particular, when $b < 3a^2/8$ and $C_2 < c  < C_1$, then $\Delta_3$ will be positive and $p_4'(x)$ will have three distinct real roots. Hence, the quartic will have three distinct stationary points $E_1$, $E_2$, and $E_3$ (see Figure 4) and it will have either four distinct real roots, or four real roots with two of them equal, or two pairs of equal real roots --- {\bf Theorem 4, Theorem 5,} and {\bf Theorem 6}, respectively. When $c = C_0$, the side $M\!P$ of the Siebeck--Marden--Northshield triangle will be parallel to the abscissa. \\
In the case of  $b < 3a^2/8$ and $c = C_{1,2}$, the discriminant $\Delta_3$ will be zero and $p_4'(x)$ will have three real roots, two of which --- equal. The Siebeck--Marden--Northshield triangle will have a side, perpendicular to the $x$-axis: $N\!P$, if $c = C_1$, or $M\!N$, if $c = C_2$. As a result, the quartic will have a treble real root and a single real root ({\bf Theorem 7}). The converse is also true: if the quartic has a treble real root, then the Siebeck--Marden--Northshield triangle will have a side perpendicular to the abscissa. \\
The stationary points of the polynomial $p_4'(x)$, that is, the two points of inflection $I_1$ and $I_2$ of $p_4(x)$ (see Figure 4) have $x$-coordinates given by the two roots of $p_4''(x) = 12 x^2 + 6 a x + 2b$:
\b
\rho_{1,2} = -\frac{a}{4} \, \pm \, \frac{\sqrt{3}}{12} \, \sqrt{3a^2 - 8b}.
\e
On Figure 4, these are points $K$ and $J$, respectively. \\
The inflection point of the cubic polynomial $p_4'(x)$ is $-a/4$ --- the root of $p_4'''(x) = 24x + 6a$. Thus, the inscribed circle of the Siebeck--Marden--Northshield triangle is centred at $(-a/4, 0, 0)$ (point $H$ on Figure 4). The radius of the inscribed circle is $R = (\sqrt{3}/12) \, \sqrt{3a^2 - 8b}$ --- same as the radius of the inscribed sphere of the regular tetrahedron, whose section with the $xy$-plane is the inscribed circle. \\
The side of the equilateral Siebeck--Marden--Northshield triangle has length $l = \sqrt{12} R = (1/2) \sqrt{3a^2 - 8b}$. The edge of the regular tetrahedron is $L = \sqrt{24} R = (\sqrt{2}/2) \sqrt{3a^2 - 8b}$. Clearly, one has $L = \sqrt{2} l$. \\
When $c = C_{1,2}$, the Siebeck--Marden--Northshield triangle has a side perpendicular to the abscissa. The roots of the cubic $p_4'(x)$ in these two cases are the double root $\rho_{1,2}$ and the single root
\b
\phi_{1,2} = - \frac{3a}{4} - 2 \rho_{1,2} = -\frac{a}{4} \, \mp \, \frac{\sqrt{3}}{6} \, \sqrt{3a^2 - 8b},
\e
respectively.
\begin{theorem}
$\phantom{EMP}$ \\
$\phantom{M}$ (1) The maximum length of the interval containing the four real roots of the quartic $x^4 + a x^3 + b x^2 + c x + d$ is equal to the length $L$ of the edge of the tetrahedron: $L = \sqrt{24} R = (\sqrt{2}/2) \sqrt{3a^2 - 8b}$. \\
$\phantom{L \,\,}$ (2) The minimum length of this interval is equal to the height $h$ of the tetrahedron: $h = 4R = (\sqrt{3}/3) \sqrt{3a^2 - 8b}$.
\end{theorem}

\n
Thus, the four real roots of the quartic lie in an interval of length between approximately $0.577 \sqrt{3a^2 - 8b}$ and approximately $0.707 \sqrt{3a^2 - 8b}$.
\begin{proof} As the roots of the quartic are the $x$-coordinate projections of the vertices of a tetrahedron, which is of fixed size when $a$ and $b < 3a^2/8$ are fixed, variation of the coefficients $c$ and $d$ of the quartic (while, of course, keeping $c$ between $C_2$ and $C_1$ inclusive for four real roots to exist) will only result in rotation of the tetrahedron around the $y$-axis and the $z$-axis (rotation around the $x$-axis leaves the roots invariant). Note that, with the variation of $c$ (again, between $C_2$ and $C_1$ inclusive), the Siebeck--Marden--Northshield triangle also rotates in the $xy$-plane around its centroid, so that the $x$-coordinate projections of its vertices alternate with the $x$-coordinate projections of the vertices of the tetrahedron --- as one cannot have two roots of the quartic between a minimum and a maximum, that is, for a {\it quartic} polynomial with four real roots, there could be only one extremum between two neighbouring roots. Of course, in the case of a double middle root, it coincides with the local maximum, while in the case of a triple root, it coincides with the saddle point of the quartic. And these are exactly the two extreme cases in the premise of this theorem. \\
$\phantom{E}$ {\it (1)} The four real roots of the quartic will lie in an interval of maximum length when they are at length $L = (\sqrt{2}/2) \sqrt{3a^2 - 8b}$ apart, as this is the maximum possible length of the $x$-coordinate projection of the tetrahedron upon rotations about the $y$-axis and the $z$-axis. Namely, one edge of the tetrahedron must have $x$-coordinate projection with the same length $L$, that is, the tetrahedron must have an edge parallel to the abscissa. There are two situations in which this can be realized. One can either have: \\
$\phantom{EMP}$ {\it (a)} A double root of the quartic exactly at the $x$-coordinate projection $-a/4$ of the centroid of the tetrahedron (which coincides with the centroid of Siebeck--Marden--Northshield triangle) and two simple roots on either side of the double root --- at distance $L/2$. Hence, the four real roots of the quartic in this case are $x_{1,4} = -a/4 \pm (\sqrt{2}/4) \sqrt{3a^2 - 8b}$ and $x_{2,3} = -a/4$. \\
This is achieved when $3a^2 - 8b > 0, \,\, c = C_0$, and $d = \widetilde{d}$, that is, the third case of {\bf Theorem 5}. \\
$\phantom{EMP}$ {\it (b)} Two pairs of double roots of the quartic (in this case, the projection of the tetrahedron onto the $xy$-plane is a square of side $L$, with its diagonals drawn). Both pairs of equal roots are equidistant from the centroid: $x_1 = x_2 = -a/4 + (1/4)\sqrt{3a^2 - 8b} = -a/4 + (\sqrt{2}/4)L = -a/4 + l/2$ and $x_3 = x_4 = -a/4 - (1/4)\sqrt{3a^2 - 8b} = -a/4 - (\sqrt{2}/4)L = -a/4 - l/2$. [The Vi\`ete formul\ae \, for a quartic with two pairs of equal real roots are $2(x_1 + x_2) = -a, \,\, (x_1 + x_2)^2 + 2 x_1 x_2 = b, \,\, 2 x_1 x_2 (x_1 + x_2) = -c,$ and $x_1^2 x_2^2 = d$. From the first two of these, one gets $x_1 = [1/(2x_2)] (b - a^2/4)$ to eliminate $x_1$ from $2(x_1 + x_2) = -a$ and get the quadratic equation $2 x_2^2 + a x_2 + b - a^2/4 = 0$ from which all roots are determined.]
\\
This is achieved when $3a^2 - 8b > 0, \,\, c = C_0$, and $d = d^\dagger$ --- the case of {\bf Theorem 6}. \\
The Siebeck--Marden--Northshield triangle in both cases {\it (a)} and {\it (b)} has a side parallel to the abscissa (as $c = C_0$) and the three real roots of $p_4'(x)$ are $\sigma_2 = -a/4$ and
\b
\sigma_{1,3} = -\frac{a}{4} \pm \frac{l}{2} =  -\frac{a}{4} \pm \frac{1}{4} \sqrt{3a^2 - 8b}
\e
--- the latter equidistant form $\sigma_2$ (which lies on the centroid). \\
Note that in case {\it (b)}, the double roots of the quartic coincide with the smallest and the largest root of $p_4'(x)$, that is, two of the vertices of the tetrahedron have $x$-coordinate projections equal to that of one the vertices of the equilateral triangle, while the other two vertices of the tetrahedron have $x$-coordinate projections equal to that of another vertex of the triangle. The remaining vertex of the triangle projects onto the $x$ axis in the middle of the projections of the other two --- exactly where the centroid of the tetrahedron projects. \\
$\phantom{E}$ {\it (2)} The four real roots of the quartic will lie in an interval of minimum length when the $x$-coordinate projection of the tetrahedron has a minimum length. Clearly, this minimum length is the height $h = (\sqrt{3}/3) \sqrt{3a^2 - 8b}$ of the tetrahedron. Therefore, the quartic will have either a treble root $-a/4 - R = - a/4 - (\sqrt{3}/4) \sqrt{3a^2 - 8b}$ and a single root $-a/4 + 3R = -a/4 + (3 \sqrt{3}/4)\sqrt{3a^2 - 8b}$ (when $c = C_2$) or it will have a single root $-a/4 + 3R = -a/4 - (3 \sqrt{3}/4)\sqrt{3a^2 - 8b}$ and a treble root  $-a/4 - R = - a/4 - (\sqrt{3}/4) \sqrt{3a^2 - 8b}$ when $c = C_1$. \\
The Siebeck--Marden--Northshield triangle in this case has a side perpendicular to the abscissa and the three real roots of $p_4'(x)$ are either the double root $-a/4 - R = - a/4 - (\sqrt{3}/4) \sqrt{3a^2 - 8b} $ and the single root $-a/4 + 2R = -a/4 + (\sqrt{3}/2)\sqrt{3a^2 - 8b}$ (when $c = C_2$) or the double root $-a/4 + R = - a/4 + (\sqrt{3}/4) \sqrt{3a^2 - 8b} $ and the single root $-a/4 - 2R = -a/4 - (\sqrt{3}/2)\sqrt{3a^2 - 8b}$ (when $c = C_1$). \\
Clearly, this occurs only when  $3a^2 - 8b > 0, \,\, c = C_{1,2}$, and $d = d^\dagger$ --- the case of {\bf Theorem 7}.
\end{proof}
\begin{theorem}
A quartic equation with four real roots cannot have a root greater than $\lambda_{max} = -a/4 + (\sqrt{3}/4) \, \sqrt{3a^2 - 8b}$ and cannot have a root smaller than $\lambda_{min} = -a/4 - (\sqrt{3}/4) \, \sqrt{3a^2 - 8b}$.
\end{theorem}
\begin{proof}
As the $x$-coordinate projection of the centroid of the tetrahedron is at $-a/4$ and as the biggest distance within the tetrahedron from its centroid is $3R = (\sqrt{3}/4) \, \sqrt{3a^2 - 8b}$ (along the height), then, clearly, a quartic equation with four real roots cannot have a root farther than $3R = (\sqrt{3}/4) \, \sqrt{3a^2 - 8b}$ from $-a/4$. Hence, no real root of a quartic can be bigger than $\lambda_{max} = -a/4 + 3R = -a/4 + (\sqrt{3}/4) \, \sqrt{3a^2 - 8b}$ and no real root of the quartic can be smaller than $\lambda_{min} = -a/4 - 3R = -a/4 - (\sqrt{3}/4) \, \sqrt{3a^2 - 8b}$.
\end{proof}
\n
Note that when the quartic has a root which is at $-a/4 \pm 3R = -a/4 \pm (\sqrt{3}/4) \, \sqrt{3a^2 - 8b}$, then each of the other three roots of the quartic is equal to $-a/4 \mp R = -a/4 \mp (\sqrt{3}/12) \, \sqrt{3a^2 - 8b}$. This falls under {\bf Theorem 7}: $\,\,\, b < \frac{3}{8} a^2, \,\,\, c = C_{1,2}$, and $d = d^\dagger$. \\
Also, note that there are no quartics that simultaneously have $\lambda_{min}$ and $\lambda_{max}$ as roots.

\section{Localization of the Roots of the Quartic Equation}

As the roots of the quartic and its stationary points (namely, the $x$-coordinate projections of the vertices of the tetrahedron and those of the vertices of the Siebeck--Marden--Northshield triangle) alternate, the localization of the four real roots of the quartic can be easily done. \\
The two cases of {\bf Theorem 8}, following under the premises of the third case of {\bf Theorem 5} and {\bf Theorem 6}, on the one hand, and {\bf Theorem 7}, on the other, are such that the roots can be found explicitly and hence, these cases will not have to be considered. \\
What remains to be considered are those that fall into the premises of {\bf Theorem 4} and the first two cases of {\bf Theorem 5}. \\
The isolation intervals of the three real roots $\widehat{x}_1 > \widehat{x}_2 > \widehat{x}_3$ of the cubic $p_4'(x) = 4x^3 + 3ax^2 + 2bx + c$ are:
\begin{itemize}
\item [(i)] For $C_2 \le c \le C_0: \quad$ $\sigma_3 \le \widehat{x}_3 \le \rho_2, \quad \rho_2 \le \widehat{x}_2 \le -a/4, \quad$ and  $\quad \sigma_1 \le \widehat{x}_1 \le \phi_2$.
\item [(ii)] For $C_0 \le c \le C_1: \quad$ $\phi_1 \le \widehat{x}_3 \le \sigma_3, \quad -a/4 \le \widehat{x}_2 \le \rho_1, \quad$ and  $\quad \rho_1 \le \widehat{x}_1 \le \sigma_1$.
\end{itemize}
    Consider first $C_2 \le c \le C_0$. Under {\bf Theorem 9}, the smallest possible root of any quartic with four real roots is $\lambda_{min} = -a/4 - (\sqrt{3}/4) \sqrt{3a^2 - 8b}$ and the biggest possible root of any quartic with four real roots is  $\lambda_{max} = -a/4 + (\sqrt{3}/4)\sqrt{3a^2 - 8b}$ (when $c = C_2$). The smallest possible root of the cubic $p_4'(x)$ (the left minimum of the quartic) is smaller than or equal to $\rho_2$, hence the smallest root $x_4$ of the quartic is greater than or equal to $\lambda_{min}$ and smaller than or equal to $\rho_2$. The next root $x_3$ of the quartic is between the left minimum and the maximum, that is, it is greater than or equal to $\sigma_3$ and smaller than or equal to $-a/4$. The root $x_2$ is between the maximum and the right minimum, namely, it is greater than or equal to $\rho_2$ and smaller than or equal to $\phi_2$. Finally, the biggest root $x_1$ of the quartic is greater than or equal to $\sigma_1$ and smaller than or equal to $\lambda_{max}$. Namely:
\b
&& \hskip-.75cm\mbox{For} \,\, C_2 \le c \le C_0: \,\, \lambda_{min}  \le x_4 \le \rho_2, \,\,\, \sigma_3 \le x_3 \le -\frac{a}{4}, \,\,\, \rho_2 \le x_2 \le \phi_2, \,\,\, \sigma_1 \le x_1 \le \lambda_{max}. \nonumber \\ &&
\e
In a similar manner:
\b
&& \hskip-.75cm\mbox{For} \,\, C_0 \le c \le C_1: \,\, \lambda_{min}  \le x_4 \le \sigma_3, \,\,\, \phi_1 \le x_3 \le \rho_1, \,\,\, -\frac{a}{4} \le x_2 \le \sigma_1, \,\,\, \rho_1 \le x_1 \le \lambda_{max}. \nonumber \\ &&
\e
As there is an overlap in the above intervals, these cannot be referred as {\it isolation intervals} of the roots of the quartic --- as two roots can occur in any one of them. \\
Consider, as first example, the quartic $x^4 + 3 x^3 + 2 x^2 - x - 19/20$. For it, one has: $a = 3, \,\, b = 2 < 3a^2/8, \,\, c = 2$ which is between $C_2 = -1.2526$ and $C_0 = -0.3750$ (also, $C_1 = 0.5026$) and $d = -19/20$ which is between $d_3 = -1.000$ and $d_2 = -0.9288$ (also, $d_1 = 0.0967$). Hence, under {\bf Theorem 4}, there are four distinct real roots. Indeed, these are: $x_1 = 0.6094, \,\, x_2 = -0.7928, \,\,  x_3 = -1.2787,$ and $x_4 = -1.5379$. One further has: $L = 2.3452, \,\, R = 0.4787, \,\, \lambda_{min} = -5.0584, \,\, \rho_2 = -1.2287, \,\, \sigma_3 = -1.5792, \,\, -a/4 = -0.7500, \,\, \phi_2 = 0.2074, \,\, \sigma_1 = 0.0792,$ and $\lambda_{max} = 3.5584.$ Clearly, the roots are within their prescribed intervals. \\
As second example, the quartic $x^4 - 4 x^3 + 5 x^2  - (7/4) x - 1/5$ has  $a = -4, \,\, b = 5 < 3a^2/8, \,\, c = -7/4$ which is between $C_0 = -2.0000$ and $C_1 = -1.4557$ (also, $C_2 = -2.5443$) and $d = -1/5$ which is between $d_3 = -0.2659$ and $d_2 = -0.1681$ (also, $d_1 = 0.1840$). Again, under {\bf Theorem 4}, there are four distinct real roots and they are: $x_1 = 1.7679, \,\, x_2 = 1.4535, \,\, x_3 = 0.8682,$ and $x_4 = -0.0896$. Additionally: $L = 2.0000, \,\, R = 0.4082, \,\, \lambda_{min} = -2.6742, \,\, \sigma_3 = 0.2929, \,\, \phi_1 = 0.1835, \,\, \rho_1 = 1.4082, \,\, -a/4 = 1.0000, \,\, \sigma_1 = 1.7071,$ and $\lambda_{max} = 4.6742$ --- again, the roots are within their prescribed intervals. \\
Similar considerations can be extended to the cases of a quartic with only two real roots, when the Siebeck--Marden--Northshield triangle still exists, but the tetrahedron does not --- one can bind the two real roots with the help of the bounds of the stationary points of the quartic (the vertices of the Siebeck--Marden--Northshield triangle), which are invariant under the variation of the free term $d$, and the nearest roots of the quartic when it has four real roots, i.e. simply within different ranges of $d$.

\section{Towards the Quintic Equation}

\n
One can use the presented in this work complete root classification of the quartic in order to determine the number of stationary points of the quintic by solving the quartic equations which stem from the different cases of the complete root classification, and which are with known number of real roots for different ranges of the coefficients of the quartic. Hence one can find the number of times the discriminant of the quintic changes sign by variation of its free term. This will lead to an improvement of the complete root classification for the quintic, proposed (by other means) in \cite{ch}; see also \cite{28} for the localization of the roots of the quintic with the help of the roots of some specific quadratic equations. For equations of degree $n \ge 6$, this method is no longer applicable, as it is not possible to find their stationary points which are roots of equations of degree $n-1 \ge 5$. \\
The discriminant the quintic polynomial $x^5 + p x^4 + q x^3 + r x^2 + s x + t$ is quartic in the free term $t$:
\b
&& \hskip-0.5cm\Delta_5 = 3125 \, t^4 + (256 p^5 - 1600 p^3 q + 2000 p^2 r + 2250 p q^2 - 2500 p s - 3750 q r) \, t^3 \nonumber \\
& & \hskip-.3cm + (-192 p^4 q s - \!128 p^4 r^2 + 144 p^3 q^2 r - \!27 p^2 q^4 + \!160 p^3 r s + 1020 p^2 q^2 s + \!560 p^2 q r^2 + \!108 q^5 \nonumber \\
& & \hskip-.12cm -  630 p q^3 r - 50 p^2 s^2 - 2050 p q r s - 900 p r^3 - 900 q^3 s + 825 q^2 r^2 + 2000 q s^2 + \!2250 r^2 s) \, t^2 \nonumber \\
& & \hskip-.3cm+ (144 p^4 r s^2 - \!6 p^3 q^2 s^2 - 80 p^3 q r^2 s + \!16 p^3 r^4 + \!18 p^2 q^3 r s - \!4 p^2 q^2 r^3 - \!36 p^3 s^3 - 746 p^2 q r s^2 \nonumber \\
& &  \hskip-.12cm  + 24 p^2 r^3 s + 24 p q^3 s^2 + 356 p q^2 r^2 s - 72 p q r^4 - 72 q^4 r s + 16 q^3 r^3 + 160 p q s^3 \nonumber \\
& & \hskip-.12cm + 1020 p r^2 s^2 + 560 q^2 r s^2 - 630 q r^3 s + 108 r^5 - 1600 r s^3) \, t \nonumber \\
& & \hskip-.3cm - 27 p^4 s^4 + 18 p^3 q r s^3 - 4 p^3 r^3 s^2 - 4 p^2 q^3 s^3 + p^2 q^2 r^2 s^2 + 144 p^2 q s^4 - 6 p^2 r^2 s^3 - 80 p q^2 r s^3 \nonumber \\
& & \hskip-.3cm + 18 p q r^3 s^2 + 16 q^4 s^3 - 4 q^3 r^2 s^2 - 192 p r s^4 - 128 q^2 s^4 + 144 q r^2 s^3 - 27 r^4 s^2 + 256 s^5. \nonumber \\
\e
The discriminant of this quartic in $t$ is
\b
& & \hskip-0.5cm \Delta_t = - 256 \, [8000 \, s^3 + (- 1408 p^4 + 7040 p^2 q - 9600 p r - 5200 q^2) \, s^2 \nonumber \\
& & \hskip+1.6cm + \, (64 p^8 - 640 p^6 q + 896 p^5 r + 2064 p^4 q^2 - 4192 p^3 q r - 2392 p^2 q^3
\nonumber \\
& & \hskip+2.1cm + \, 3120 p^2 r^2 + 2000 p q^2 r + 1120 q^4 + 1800 q r^2) \, s \nonumber \\
& & \hskip+1.6cm - \, 32 r q p^7 + 8 p^6 q^3 + 16 p^6 r^2 + 288 p^5 q^2 r - 69 p^4 q^4 - 568 p^4 q r^2 - 660 p^3 q^3 r \nonumber \\
& & \hskip+1.6cm  + \, 168 p^2 q^5 + 208 p^3 r^3 + 2234 p^2 q^2 r^2 + 80 p q^4 r - 80 q^6 - 2340 p q r^3
\nonumber \\
& & \hskip+1.6cm - \, 440 r^2 q^3 + 675 r^4]^2 \nonumber\\
& & \hskip+1.45cm \times [- 2000 \, s^3 + (432 p^4 - 2160 p^2 q + 2400 p r + 1800 q^2) s^2 \nonumber \\
& & \hskip+1.90cm + \,  (-432 p^3 q r + 108 p^2 q^3 + 120 p^2 r^2 + 1800 p q^2 r - 405 q^4 - 2700 q r^2) \, s \nonumber \\
& & \hskip+1.90cm + \, 128 p^3 r^3 - 36 p^2 q^2 r^2 - 540 p q r^3 + 135 r^2 q^3 + 675 r^4]^3.
\e
Clearly, the sign of $\Delta_t$ is opposite to the sign of the cubed term (the last pair of square brackets on the last three lines in the formula). This term is cubic in $s$ and its discriminant is
\b
\widetilde{\Delta}_{s} = -5038848 (4 p^3 - 15 p q + 25 r)^2 (8 p^3 r - 3 p^2 q^2 - 30 p q r + 10 q^3 + 25 r^2)^3.
\e
The sign of $\widetilde{\Delta}_{s}$ depends again on a cubed term (the last pair of parenthesis), which is quadratic in $r$. The discriminant of this quadratic in $r$ term is
\b
\widetilde{\Delta}_{r} = 8 (2 p^2 - 5 q)^3
\e
and hence $\widetilde{\Delta}_{s}$ can be written as
\b
\widetilde{\Delta}_{s}  = -5038848 (r - R_0)^2 [(r - R_2)\,(r - R_1)]^3,
\e
where $R_0 = -4 p^3/25 + 3 p q/5$ and
\b
R_{1,2} = R_0 \pm \frac{\sqrt{2}}{25} \, \sqrt{(2 p^2 - 5 q)^3}.
\e
Analysis, similar to the one in Sections 3 and 4, should be performed to reveal the Complete Root Classification for the quintic.

\section{Discussion}
The essence of the geometrical study of polynomial roots, presented in this paper, is based on an approach rooted in the analysis of the ``discriminant of the discriminant". It can be summarized as follows. Polynomials can be viewed as elements of congruence of curves representing same-degree polynomials, the graphs of which foliate the $xy$-plane by the continuous variation of some foliation parameter. Most conveniently, one could choose the free term as such foliation parameter. All polynomials in this congruence have the same set of stationary points. Among these polynomials, there are those for which the abscissa is tangent to their graph at the stationary point. These polynomials have vanishing discriminants and they can be viewed as ``boundaries" setting ‘‘polynomial bands’’ in such a manner that within each polynomial band, all polynomials have the same number of real roots and the polynomial discriminants have the same sign. The variation of the foliation parameter, "shifts" through the different bounds, and the Complete Root Classification can be done with this technique. This is a novel method and the Complete Root Classification obtained in such a way is a significant improvement of those available in the literature --- as new features are uncovered and finer structure of the classification is revealed. Moreover, the conditions for the various cases are conditions on the individual parameters of the polynomial, rather than their intricate combinations or the discriminant. For example, consider the Sturmian constant $S_4 = -3 a^3 c + (b^2 - 6d) a^2 + 14 a b c - 4 b^3 + 16 b d - 18 c^2$ for the quartic (see the beginning of Section 3) from Cayley's classification from 1861. In an attempt to determine its sign, note that $S_4$ is cubic in $a$ and cubic in $b$ and each of these two cubics have very complicated discriminants. The Sturmian constant  $S_4$ is also quadratic in $c$ with discriminant simplifying to $(3 a^4 - 20 a^2 b + 36 b^2 - 144 d)(3 a^2 - 8 b)$ and not being easy to analyse. It is linear in the free term $d$ with root given by $-(1/2)(3 a^3 c - a^2 b^2 - 14 a b c + 4 b^3 + 18 c^2)/(3 a^2 - 8 b)$. Knowing whether $d$ is less than, equal to, or greater than this root, sheds no light on the sign of the Sturmian constant $S_5$, which is the discriminant $\Delta$ of the quartic.

\section{Conclusions} With the proposed method, the different polynomial bands for the symbolic quartic are determined by different intervals for the free term and the boundary points of these intervals are roots of cubic equations. These cubic equations are themselves symbolic and the determination of their roots may prove to be difficult. This limitation however can be partially lifted using the presented method for these symbolic cubic polynomials. This allows the determination of the isolation intervals of their roots and hence advances the Complete Root Classification of the quartic. \\
It will be interesting to study in the future polynomials of higher degrees through the recursive determination of the discriminant of the discriminant of the polynomial --- in its consecutive coefficients, starting with the free term. In view of the great simplifications occurring, one can look for underlying structures.

\section*{Acknowledgements}
It is a great pleasure to thank Jacques G\'elinas and Sam Northshield for their very valuable comments.

\end{document}